\documentclass[11pt]{amsart}

\usepackage{amssymb,amsmath,amsthm}
\usepackage{color}
\makeindex
\usepackage{enumerate}

\usepackage[pagebackref,colorlinks,linkcolor=red,citecolor=blue,urlcolor=blue,hypertexnames=true]{hyperref}

\newtheorem{thm}{Theorem}[section]
\newtheorem{theorem}[thm]{Theorem}
\newtheorem{corollary}[thm]{Corollary}
\newtheorem{lemma}[thm]{Lemma}
\newtheorem{proposition}[thm]{Proposition}
\newtheorem{assumption}[thm]{Assumption}
\newtheorem{remark}[thm]{Remark}

\newtheorem{definition}[thm]{Definition}

\theoremstyle{remark}

\def\beq{ \begin{equation}}
\def\eeq{  \end{equation} }
\def\bes{\begin{equation*}}
\def\ees{\end{equation*}}
\def\bmat{\begin{bmatrix}}
\def\emat{\end{bmatrix}}
\def\barr{\begin{array}}
\def\earr{\end{array}}

\def\bet{\begin{theorem}
}
\def\eet{\end{theorem}}

\def\Nd{ {[\frac{d}{2}]_+ }}

\newcommand{\RR}{{\mathbb R}}

\def\ncbo{{free basic open\ }}
\def\ncbbo{{free bounded basic open\ }}

\def\Ndsig{d}
\def\LC{\Lambda}

\def\dd{\delta}

\def\bd{\partial \mathcal D}

\def\hbd{\widehat{\bd}}
\def\cM{\mathcal M}
\def\cB{\mathcal B}
\def\cC{\mathcal C}
\def\cD{\mathcal D}
\def\cM{\mathcal M}

\def\PM{P_{\mathcal M}     }

\def\hL{\hat{L}}
\def\cL{\mathcal L}

\def\cS{\mathcal S}

\def\cI{\mathcal I}

\def\cP{\mathcal P}
\def\ccP{\mathcal P}

\def\cW{\mathcal W}
\def\cZ{\mathcal Z}

\def\cN{\mathcal N}

\def\smatng{\mathbb S_n(\mathbb R^g)}

\def\smatg{\mathbb S(\mathbb R^g)}

\def\posI{\mathfrak{I}}
\def\posn{\mathfrak{I}_p(n)}
\def\pos{\mathfrak{I}_p}

\def\cDpn{\cD_p(n)}

\def\cPN{\cP_\Nd}

\def\cPNd{\cPN^\dd }
\def\cPNpd{ \cP_{\Nd+1}^\dd }
\def\cPNsig{ \cP_d }

\def\cPddd{ \cP_d^{\dd \times \dd} }

\def\cPNsigd{ \cPNsig^\dd }

 \def\mfS{\mathfrak{S}}

\def\SS{S}

\def\cDpn{\cD_p(n)}

\def\dddd{\delta \times \delta}
\def\ddp{\dd^\prime}

\def\pnclos{ $\cPNsigd$-closure}
\def\pnclosed{$\cPNsigd$-closed}

\def\rds{respects direct sums}

\def\nus{\breve \nu}
\def\cMs{{\breve M}}
\def\PMs{P_{\cMs}}

\def\ocD{\overline{\cD}}

\def\ttq{{\tt{q}}}

\def\mfP{\mathfrak{P}}

\newcommand{\df}[1]{{\bf{#1}}{\index{#1}}}

\def\star{T} 
\def\st{T}  

\numberwithin{equation}{section}

\begin{document}

\setcounter{page}{1} \title[Convex Free Semi-Algebraic Sets]{Every Convex Free
Basic Semi-Algebraic Set has an LMI Representation}

\author[Helton]{J. William Helton${}^1$}
\address{Department of Mathematics\\
  University of California \\
  San Diego 92093}
\email{helton@math.ucsd.edu}
\thanks{${}^1$Research supported by NSF grants
DMS-0700758, DMS-0757212, and the Ford Motor Co.}
\author[McCullough]{Scott McCullough${}^2$}
\address{Department of Mathematics\\
  University of Florida 
   }
   \email{sam@math.ufl.edu}
\thanks{${}^2$Research supported by the NSF grant DMS-0758306.}



\date{\today}

\maketitle

\begin{abstract}
   The (matricial) solution set of a
  Linear Matrix Inequality (LMI)  is
  a convex  free basic open semi-algebraic  set
 (discussed below).
  The main theorem of this paper is a converse,
  a result
  which has implications for semi-definite programming
  and systems engineering as well as free semi-algebraic geometry. 

  A free basic open semi-algebraic set is defined in terms of a
  symmetric free
 $\dd \times\dd$ matrix-valued  polynomial $p(x_1, \cdots, x_g).$
 Such a polynomial
 is a linear combination of words in freely non-commuting
   variables $(x_1,\dots,x_g)$ with coefficients
  from $M_\dd$, the $\dd \times\dd$ matrices over $\mathbb R$.
  The involution ${}^{\st}$
   on words given by sending a concatenation of letters
   to the same letters, but in the reverse order
   (for instance $(x_jx_\ell)^{\st}=x_\ell x_j),$ extends
   naturally to such polynomials and $p$
    is itself symmetric if $p^{\st}=p$.
   Let $\mathbb S_n(\mathbb R^g)$ denote
   the set of $g$-tuples $X=(X_1,\dots,X_g)$
  of symmetric $n\times n$ matrices.
   The  polynomial $p$ is naturally evaluated on a tuple
    $X\in\mathbb S_n(\mathbb R^g)$
    yielding
    a value $p(X)$
    which is a $\dd \times \dd$ block matrix with $n\times n$
    matrix entries.
    Evaluation
    at $X$ is compatible with the involution since
    $p^{\st}(X)=p(X)^{\star}.$ In particular, if $p$ is symmetric,     
    then  $p(X)$ is a symmetric matrix.

  Assuming that $p(0)$ is invertible, the invertibility
    set $\cD_{p}(n)$ of
  a free symmetric polynomial $p$ in dimension $n$
  is the component of $0$ of the set
\[
   \{ X \in\mathbb S_n(\mathbb R^g) \; : \
  \ p(X) \mbox{ is invertible} \}.
\]
  The invertibility
    set, $\cD_p$, is the
  sequence of sets $(\cD_p(n) ).$ It 
  is an example of a
  \ncbo semi-algebraic set.
 The sequence $\cD_p$ is convex if $\cD_p(n)$ is convex for each $n$.
   When $p=L$ is an affine linear symmetric 
  polynomial with constant term $I_\delta$,
  the expression $L(X)\succ 0$ is a linear matrix inequality and, as is clear,
   $\cD_L$ is a sequence of convex sets. 

 The main theorem of this article implies:
 {\it if $p(0)$ is invertible and  $\cD_p$
 is  bounded,
 then there is an $\ell$ and an affine linear $L$
  of size $\ell$ with constant term $I_\ell$  
  such that $\cD_{p}(n)=\cD_L(n)$ for each $n$ 
  if and only if $\cD_p$ is convex.}
\end{abstract}

\section{Introduction}
 The main result of this paper is the free algebra analog of
  the preposterous statement: 
\vspace{+.05in}

\centerline{\it A bounded open  convex set $\cC$ in $\mathbb R^n$ with
algebraic boundary is a polytope.}

\vspace{+.05in}

 A recurring theme in the
 non-commutative setting, such as that of
 a subspace of  C-star algebra
 \cite{Ar69,Ar72,Ar08} or in
 free probability \cite{Vo04,Vo05}
 to give two of many examples, is the need to consider
 the {\it complete matrix structure} afforded by
 tensoring  with $n\times n$ matrices (as $n$ ranges over all 
  positive integers). 
 The resulting theory of operator
 algebras, systems, spaces, and matrix
 convex sets 
 has matured to the point that there are now
 several excellent books on the subject including
 \cite{BL04} \cite{Pa02} \cite{Pi03}.

 Here we consider a bounded matrix convex set 
 $\cD_p$ which is free semi-algebraic in the sense that it is 
 determined by a matrix-valued free polynomial and show that
 there is a monic affine linear matrix-valued polynomial $L$ such 
  that $\cD_p=\cD_L$.   
 Since it involves matrix convex sets, it is not surprising
 that 
 the analysis hinges on the  matricial version
 of the Hahn-Banach Separation  Theorem of
 Effros and Winkler \cite{EW97} which
 says that given a point $x$
 outside a matrix convex set $\cC$ there is a 
 monic affine linear matrix-valued polynomial  which
 separates $x$ from $\cC$.  For a general matrix convex
 set $\mathcal C$, the conclusion is then that there is a collection,
 likely infinite, of monic affine linear matrix-valued
 polynomials which cut out $\mathcal C$.
  In the case $\mathcal C$ is matrix convex and also semi-algebraic,
  the challenge, successfully dealt with in this paper,
  is to prove that there is actually
  a single  monic affine linear 
   matrix-valued polynomial $L$  which defines $\mathcal C$; i.e.,
   $\cD_L=\cD_p$. 

    The article also contains some further results.
    For instance,  a corollary of the
    results of  Section \ref{sec:irreducible} is that if 
    $p$ satisfies certain irreducibility type hypotheses, 
    in addition to the assumption that $\cD_p$ is bounded and matrix convex,
    then $p$ has degree two. 
        In Section \ref{sec:freesag},  implications for
   free real algebraic geometry, a recently emerging 
  non-commutative analog of the classical    subject,
  are discussed. 
  Classically projections of semi-algebraic sets are semi-algebraic.
   A consequence of our main result is that this projection property is 
  false in the free case.  
  
  The main result also bears on
    a free analog of semi-definite
  programming,
  a major branch of convex optimization.
 Fundamental in  semi-definite programming is the class of  
  convex sets $C$ which  can be represented with an
  LMI, as is the much more general class consisting of projections 
  of LMI representable sets.
  Their free analogs behave 
  very differently:  the classes of projected  LMI representable
  sets which are free semi-algebraic and LMI representable sets are the same.
  This is shown in Section \ref{sec:persp}.

The remainder of this introduction contains a
    precise statement of the main result
    preceded by the relevant definitions.

\subsection{Free Polynomials}
  \label{subsec:ncpolys}
   Let $g$ be a positive integer which is now fixed
   for the remainder of the paper.  Let
   $\ccP$ \index{$\ccP$} denote the real free algebra of polynomials in the
   freely non-commuting indeterminates $x=(x_1,\dots,x_g).$
   Elements of $\ccP$ are
   {\bf free polynomials}  
   or often just {\bf polynomials}.
  \index{free polynomial}
   Thus, a free polynomial    $p$ is a finite linear combination,
 \begin{equation}
  \label{eq:poly}
    p=\sum p_w w,
 \end{equation}
   of words $w$ in $(x_1,\dots,x_g)$ with coefficients 
 $p_w\in\mathbb{R}$.

   There is a natural involution ${}^{\st}$ on
   $\ccP$  given by
 \begin{equation}
  \label{eq:genericp}
   p^{\st}=\sum p_w w^{\st},
 \end{equation}
  where, for a word $w$,
 \begin{equation}
  \label{eq:wordT}
   w=x_{j_1}x_{j_2}\cdots x_{j_n} \mapsto w^{\st} = x_{j_n}
   \cdots x_{j_2}x_{j_1}.
 \end{equation}
   A polynomial $p$ is symmetric if
  it is invariant with respect to the involution.
   In particular, $x_j^{\st}=x_j$ and for this reason
   the variables are sometimes referred to as symmetric
   free  variables. \index{symmetric variables}


\subsection{Evaluations}
 \label{sec:openG2}
  Let $\smatng$ denote the set of $g$-tuples \index{$\smatng$}
  $X=(X_1,\ldots,X_g)$ of real symmetric $n\times n$ matrices.
  Let $M_n$ denote the $n\times n$ matrices with real entries. 
  Each $X\in\smatng$ determines a representation 
  $e_X:\mathcal P \to M_n$  by evaluation.  Indeed, by 
  linearity, $e_X$ is determined by its action on words where
  $e_X(\emptyset)=I_n$ and for
   a non-empty word $w$ as in equation (\ref{eq:wordT}),
 \begin{equation}
  \label{eq:wX}
  e_X(w)=X_{j_1}X_{j_2}\cdots X_{j_n}.
 \end{equation}
  It is natural to write $p(X)$ instead of the more formal $e_X(p)$.
 
  Note that $p(X)$ respects the involution in the sense 
  that $p^{\st}(X)=p(X)^{\star}$.  In particular, if $p$ is symmetric,
  then so is $p(X)$.
  Finally, if $\pi:\mathcal P\to M_n$ is a representation which
  respects the involution, then there is an $X\in\smatng$ such
  that $\pi(p)=p(X)$.

\subsection{Matrix-Valued Polynomials}
  Let $\ccP^{\dd\times \ddp}$ denote \index{$\ccP^{\dd\times \ddp}$}
  the $\dd\times \ddp$ matrices with entries
  from $\ccP$.  Because row vectors of 
  polynomials figure prominently in this article, 
   $\ccP^{1\times \dd}$ is often abbreviated to 
   $\ccP^{\dd}$.

  Evaluation   at $X \in \smatng$ naturally
  extends entrywise to $p\in \ccP^{\dd\times\ddp}$
  with the result, $p(X),$ a $\dd\times\ddp$ block
  matrix with entries from $M_n$. 
  Up to unitary equivalence,
  evaluation at $X$ is conveniently described using tensor
  product notation by writing $p$ as  a finite linear combination
 \begin{equation}
  \label{pww}
   p=\sum_{w} p_w w,
\end{equation}
   where now the coefficients
   $p_w$ are $\dd\times \dd$ matrices (with real entries),
   and observing that 
 \[
   p(X)=\sum p_w \otimes w(X),
 \]
  where $w(X)=e_X(w)$ is given by equation \eqref{eq:wX}.

 The involution ${}^{\st}$ naturally extends to
  $\ccP^{\dddd}$ by
 \[
    p^{\st} = \sum_{w} p_w^{\star} w^{\st},
 \]
  for $p$ given by equation \eqref{pww}.
  A polynomial
  $p\in\ccP^{\dd\times \dd}$ is symmetric if
  $p^{\st}=p$ and in this case $p(X)=p(X)^{\star}$.

   A simple method of constructing new matrix-valued
  polynomials from old ones is from direct sums.
   \index{direct sum}
  For instance, if
  $p_j\in\ccP^{\dd_j\times \dd_j}$ for $j=1,2$, then
 \[
   p_1\oplus p_2=\begin{pmatrix} p_1 & 0\\0 & p_2 \end{pmatrix}
    \in \ccP^{(\dd_1+\dd_2)\times (\dd_1+\dd_2)}.
 \]

\subsection{Invertibility Sets}
 \label{subsec:semi-algebraic}
  A {\bf graded set} $\cS$ is a sequence
  $\cS=(\cS(n))_{n=1}^\infty$ where, for each
  $n,$ $\cS(n)\subset \smatng$. 
  The notation $\cS\subset \mathbb S(\mathbb R^g)$
  indicates that $\cS$ is a graded set. 
   The \df{principal component} of
  $\cS$, denoted $pc[\cS]$, 
 is the connected component of $0$ of $\cS$; i.e., the 
  graded set $pc[\cS]=(pc[\cS(n)])$.

  Suppose  $p\in \ccP^{\dddd}$ is symmetric.
  In particular, $p(0)$ is a  $\dddd$ symmetric
  matrix. Assuming that $p(0)$ is invertible,
  for each positive integer $n$  let
\[
  \posn = \{X\in\smatng : \ p(X) \mbox{ is  invertible}\} \subset \smatng,
\]
  and let $\pos$ denote the 
  graded set  $(\posn)_{n=1}^\infty$. 
  The {\bf invertibility set} $\cD_p$  of $p$
  is the graded set $\cD_p=pc[\pos]$.
   In Section \ref{sec:freesag} the graded set  $\cD_p$ is interpreted
  in terms of free semi-algebraic geometry.

\begin{remark}\rm
  By a simple affine linear
  change of variable the point $0\in\mathbb R^g$ can be
  replaced by $\lambda\in\mathbb R^g$. For $m>1$, 
  replacing $0\in\mathbb S_m(\mathbb R^g)$ 
  by a fixed $\Lambda\in\mathbb S_m(\mathbb R^g)$
  will require an extension of the theory.
\qed
\end{remark}

\begin{remark}\rm
  The graded set $\cD_p$ is closed with respect to unitary conjugation
  and direct sums - see Lemma \ref{lem:convexities} for the 
  precise statement. However, because the matrices involved
  are symmetric, a property not generally preserved under
  similarity,  $\cD_p$ is not a free set in the
  sense of Voiculescu \cite{Vo04} \cite{Vo05}. 
\qed \end{remark}

  The graded set $\cD_p$ is {\bf convex} \index{convex} if each
  $\cDpn$ is convex (in the usual sense).
  Similarly, $\cD_p$ is {\bf bounded} \index{bounded} if 
   there is a constant $K$ such for each $n$ and
  each $X\in \cD_p(n)$, 
  $\|X\|=\sum \|X_j\|\le K.$  

  The following
  list of conditions summarizes the
  usual  assumptions on $p$.

\begin{assumption}
 \label{assume}
  Fix $p$ a $\dd \times \dd$  symmetric matrix
  of polynomials of degree $d$ in $g$ free variables.
  Our standard assumptions are:
 \begin{enumerate}[(i)]
  \item  $p(0)$ is invertible;
  \item  $\mathcal D_p$ is bounded; and
  \item  $\mathcal D_p$ is convex.
 \end{enumerate}
\end{assumption}
 \index{Assumption \ref{assume}}

\subsection{Monic  Linear Pencils}
 \label{subsec:LMIs}
 A 
{\bf linear pencil} $L$ is an expression of the form
 \index{ linear pencil}
\begin{equation}
 \label{eq:linear-pencil}
    L(x):= A_0 + A_1 x_1 + \cdots + A_g x_g
\end{equation}
where, for some positive integer $\ell$,
 each $A_j$ is an $\ell \times \ell$ symmetric
matrix with real  entries. 
  (While linear pencil is standard usage,
  it is a bit of a misnomer.
  When the constant term  $A_0$ is non-zero, 
   a linear pencil is actually affine linear.)
  The integer $\ell$ is the {\bf size} of the pencil. 
\index{size of pencil}
  The pencil is monic if  $A_0=I$ in which case  $L$ is
 \index{monic  linear pencil}
 {\bf a monic linear pencil}.

Since a monic linear pencil (of size $\ell$) 
is an element of $\ccP^{\ell\times\ell}$ it evaluates
at a tuple $X \in \smatng$ as
\[
  L(X):= I_\ell \otimes I_n + A_1   \otimes X_1 + \cdots + A_g \otimes X_g.
\]

  For a square matrix $A$, 
  the notation $A\succ 0$ ($A\succeq 0)$ indicates that
  the symmetric matrix $A$ is positive definite (resp. positive
  semi-definite). 
  From the form of 
  the monic linear pencil $L$, it is straightforward to verify that
 its invertibility set is  the sequence
\[
  (\cD_L(n)) = (\{X \in \smatng : \ L(X) \succ 0 \})
\]
 and that each 
 $\cD_L(n)$ is convex.
 Moreover,
 \[
  (\overline{\cD_L(n)})=(\{X\in\smatng: \ L(X) \succeq 0\}.
 \]

   A {\bf Linear Matrix Inequality}, or {\bf LMI}
  for short, is an expression
  of the form $L(X)\succ 0.$
  LMIs figure prominently in  many branches of engineering and science.
 \index{Linear Matrix Inequality} \index{LMI}
  A graded subset $\cC=(\cC(n))$ of
  the graded set $\mathbb S(\mathbb R^g)$
  has a  {\bf (free) LMI representation}
  \index{LMI representation} \index{free LMI representation}
 if there is a monic  linear pencil $L$
  such that
$$
 \cC= \cD_L.
$$

  The following is the main theorem of this
  article.  A somewhat stronger version of
  the result appears later as Theorem \ref{cor:main-cor}.

\begin{theorem}
 \label{thm:main}
   If $p$ satisfies Assumption \ref{assume}, then 
   there is a monic  linear pencil $L$
   (of finite size)  
   such that $\cD_p(n)=\cD_L(n)$ for every $n$; that is,
   if $p\in\cP^{\delta\times \delta}$ is symmetric, $p(0)$ is invertible,
   and $\cD_p$ is bounded, then $\cD_p$ is convex
   if and only if the graded set $\cD_p$ has an LMI representation. 
\end{theorem}

  Results needed   for the proof of Theorem \ref{thm:main}
  occupy the paper up through 
  Section \ref{sec:hahnHello}.  The proof of
  the theorem itself appears in Section \ref{sec:proof}.
  That section also gives a 
  bound, depending only upon the degree $d,$ the number
  of variables $g,$ 
    and the (matrix) size $\delta$ of $p,$
   on the size of the linear pencil $L$
  needed to represent $\cD_p.$
  Section \ref{sec:irreducible}
  refines the main theorem by adding  irreducibility
  type hypotheses  on $p$ and concluding that 
  $p$ has degree two. 
  Implications for free real
  algebraic geometry and semi-definite programming appear
  in Section \ref{sec:freesag}.

\subsection{Acknowledgment}
  The authors thank David Zimmermann for his comments on the paper.

\section{Preliminaries}
\label{sec:bdryfacts}
  From  now through Section \ref{sec:proof}, fix
  a polynomial $p$ satisfying the conditions
  of Assumption \ref{assume}.  Thus amongst other things, 
  $p$ is $\dddd$ matrix-valued; has degree $d$;
  and is a polynomial in $g$ freely non-commuting variables.

  This section presents two basic facts for future use.
  The following lemma gives a useful criterion
  for containment in the closure
  $\ocD_p=(\overline{\cD_p(n)})$ of the graded set $\cD_p$. 

\begin{lemma}
 \label{inhboundary}
   Suppose $p\in\cP^{\dddd}$ satisfies the conditions
   of Assumption \ref{assume} and $n$ is a positive
   integer. If $X\in\smatng$,
   then $X\in\ocD_p(n)$ if and only if 
   $tX\in\cD_p(n)$ for all $0\le t<1$. 
\end{lemma}

\begin{proof}
  First suppose that $X\in\ocD_p(n)$. 
  Since $\cD_p(n)$
  is convex, so is $\ocD_p(n)$.
  Further, $\cD_p(n)$ contains $0\in \smatng.$
   Thus, $tX\in\ocD_p(n)$
  for $0\le t\le 1$.
  Moreover,
  there are only finitely many $0\le s\le 1$
  such that $p(sX)$ is not invertible because
  $p(0)$ is invertible and $p$ is a polynomial.
   Indeed, $p(sX)$ is invertible if and only if
  the non-zero polynomial $q(t)=\det(p(tX))$  
  is not zero at $s$.  
  If $0\le t <1$ and $p(tX)$
  is invertible, then $tX \in\posn$.
  To see that  $tX$ is in fact in $\cD_p(n)$,
  we argue by contradiction. Accordingly, suppose
  $tX\notin \cD_p(n).$ In this case, since
  $\posn$ is both open and  the disjoint union of its
  connected components, $tX$ is contained in some
  open set which does not meet $\cD_p(n)$.
  Thus, 
  $tX\notin \ocD_p(n)$,
  a contradiction.
 Now  $tX\in\cD_p(n)$ and
 since $\cD_p(n)$ is convex, $sX\in\cD_p(n)$ for $0\le s\le t$.
  Choosing a sequence $0<t_n<1$ converging to $1$  such that
  $p(t_nX)$ is invertible it
  now follows that $sX\in\cD_p(n)$ for $0\le s<1$.

  The converse is evident. 
\end{proof}

\begin{lemma}
 \label{lem:nochangestart}
   Let $\cC=(\cC(n))$ be a graded set with 
   $\cC(n)\subset \smatng$ for each $n$. 
   If each $\cC(n)$ is open and if 
   $L$ is a monic  linear pencil, then $L$ is
   positive definite on each $\cC(n)$ if and only if $L$
   is positive semi-definite on each $\cC(n)$.
\end{lemma}

\begin{proof}
   Suppose $L$ is positive semi-definite on $\cC(n)$. If
   $L$ is not positive definite on $\cC(n)$, then there is an 
   an $X\in\cC(n)$ such that $L(X)\succeq 0$ and  $L(X)$
   has a kernel.  In particular, there is a unit vector
   $v$ such that $L(X)v=0$. Let $q(t)=\langle L(tX)v,v\rangle$.
   Thus $q$ is affine linear  and $q(0)=1,$ whereas
   $q(1)=0$. Hence $q(t)<0$ for $t>1$ and thus
   $L(tX)\not\succeq 0$ for $t>1$.  On the other hand,
   since $\cC(n)$ is open and $X\in\cC(n)$, there is $t>1$ such that
   $tX\in\cC(n)$ which  contradicts $L(tX)\succeq 0$.
\end{proof}

\section{Dominating Points and the Boundaries of $\cD_p$}
 \label{sec:dominate}
  There are two notions, both important for what follows, 
  of the boundary of the graded set $\cD_p$. 
  The \df{(topological) boundary} of $\cD_p$,
  \index{boundary, topological} \index{$\bd_p$}
   denoted  $\bd_p$, is the
  graded set $(\bd_p(n))$ where 
  $\bd_p(n)$ is the usual topological boundary of $\cD_p(n)$.
  Let $\hbd_p(n)$ \index{$\hbd_p$} denote
  the set of pairs $(X,v)$ where $X\in\bd_p(n),$
  the vector $v$ is in  $\mathbb R^\dd \otimes \mathbb R^n$,
  and $p(X)v=0$.  The assumption  $v\ne 0$ will often be implicit.
  The graded set $\hbd_p = (\hbd_p(n))$
  is the \df{detailed boundary} of $\cD_p$.  \index{boundary, detailed}    
 The use of the term  graded set \index{graded set} for $\hbd_p,$ while
  technically different than the use of the term graded set 
 defined earlier, should cause no confusion.

   Given
  $(X^j,v^j)\in \mathbb S_{n_j}(\mathbb R^g)\times(\mathbb R^\dd \otimes
   \mathbb R^{n_j}),$ for $j=1,2,$ let
\[
 \oplus_{j=1}^2(X^j,v^j)   =( \begin{pmatrix} X^1 & 0 \\ 0 & X^2 \end{pmatrix},
    \begin{pmatrix} v^1 \\ v^2 \end{pmatrix} ).
\]
   This notion of direct sum \index{direct sum} clearly
  extends to a finite list $(X^j,v^j)$, $j=1,2,\dots, s$. 
  \index{$\oplus (X^j,v^j)$}
   A graded 
   set $S=(S(n))$ where
   $S(n)\subset \smatng \times (\mathbb R^\delta \otimes \mathbb R^n)$,
   \df{respects direct sums} 
   if $(X^j,v^j)\in S(n_j),$ for $j=1,2,\dots,s,$
  implies $\oplus_1^s (X^j,v^j)\in S(n),$ where $n=\sum n_j$.  
   It is evident that the graded set 
   $\hbd_p=(\hbd_p(n))$ respects direct sums.

   Let $\cPNsigd$ \index{$\cPNsigd$}
   denote the $1\times \delta$ (row)
   matrices with entries polynomials of degree at most $d$.
   If $X\in\smatng$ and $q\in\cPNsigd$, then
   $q(X)$ is a linear mapping from $\mathbb R^\delta\otimes \mathbb R^n$
   to $\mathbb R^n$.  Hence, if $(X,v)\in\hbd_p(n)$, then
   $q(X)v$ is defined. 
  Let $T=(T(n))$ denote 
  a non-empty graded subset of 
  the graded set $\hbd_p.$ 
  A point $(X,v)\in\hbd_p(m)$
 is a  \df{dominating point  of $T$}
  if $q(X)v=0$ implies that $q(Y)w=0$
  for every $n$ and $(Y,w)\in\hbd_p(n)$; i.e.,
  if $q$ vanishes \index{vanishes}
   at $(X,v)$, then $q$ vanishes on all of $T$. 
  Let $T_*$ \index{$T_*$} denote the dominating points of $T$. 
  Note $T_*$ need not be contained in $T$.

  \index{$\cI(S)$}
 Given a graded subset  $T=(T(n))_{n=1}^\infty$ of
  the graded set $\hbd_p$  let
\[
  \cI(T) = \{q \in \cPNsigd:  
   \  q(X)v=0 \
    \mbox{ for all } (X,v)\in T \} \subset \cPNsigd.
\]
 In the special case that $T=\{(X,v)\}$ is a singleton
  (so there is an $m$ such that $T(m)$
  has one element, and $T(n)$ is the empty set
  for all other $n$), 
  the notation  $\cI(X,v)$ is used  in place of
  the more cumbersome  $\cI(\{(X,v)\})$.
   Note that, in the case $\dd=1$, if not for the degree $d$ restriction,
  the subspace
 $\cI(T)$ would be a left ideal in $\ccP$.
  In any case, $\cI(T)$ is a subspace of $\cP^\delta_d$. 

  The following lemma follows readily from the definitions. 

\begin{lemma}
 \label{lem:yesdomin-warmup}
    Let $T=(T(n))$ be a non-empty graded subset of the graded set
   $\hbd_p$.  The point $(X,v)\in\hbd_p$
   is a dominating point of  $T$ if and only if
 \[
   \cI(X,v)\subset \cI(T).
 \]
   On the other hand, if $(X,v)\in T(n)$, then
 \[
  \cI(T)\subset \cI(X,v).
 \]
  Thus, if $(X,v)\in T(n)\cap T_*(n),$ then
 \[
  \cI(X,v)=\cI(T).
 \]
\end{lemma}

  Given graded subsets $A=(A(n))$ and $B=(B(n))$ of
  $\hbd_p$, the 
 {\bf intersection} of $A$ and $B$, denoted $A\cap B,$
  is  the graded set 
  $(A(n)\cap B(n))$. Similarly, $A$ is {\bf non-empty}
  if there is an $m$ so that $A(m)$ is non-empty.
  The following two lemmas are
  key facts about dominating points for graded sets which 
  respect direct sums.

\begin{lemma}
 \label{lem:yesdomin}
   Suppose $S=(S(n))$ is a non-empty graded subset of the graded set
  $\hbd_p.$ 
   If $S$  respects  direct sums, then
   there is an $m$ and a $(X,v)\in S(m)$ such that
\begin{equation}
 \label{eq:SIS}
  \cI(X,v) =\cI(S).
\end{equation}
  Hence $S\cap S_*$ is non-empty.
\end{lemma}

\begin{proof}
  First note that
 \[
  \cI(\SS)= \bigcap \{ \; \cI(Y,w):\  (Y,w)\in \SS\}.
 \]
  Thus, since each $\cI(Y,w)$ is a subspace
  of the finite dimensional vector space
  $\cPNsigd$, there exists an $s$ and
  $(Y_j,w_j)\in \SS(n_j)$ for $j=1,\dots,s$ such that
 \[
   \cI(\SS)=\cap_{j=1}^s \cI(Y_j,w_j).
 \]
  Let $(X,v)=\oplus (Y_j,w_j)$.
  Then $(X,v)\in \SS(m)$, where $m=\sum n_j,$ and
\begin{equation}
 \label{eq:intersect-sum}
  \cI(X,v)=\cap_{j=1}^s \cI(Y_j,w_j) =\cI(\SS).
\end{equation}
\end{proof}

\begin{lemma}
 \label{all-or-none}
  Suppose $S=(S(n))$ is a graded
  subset of the graded set $\hbd_p$ which respects direct sums
  and suppose $q\in \cPNsigd$.
  If $(X,v)\in S(n)\cap S_*(n)$ and $(Y,w)\in S(m)\cap S_*(m),$
  then $q(X)v=0$ if and only if $q(Y)w=0$;  that is,
  $q$ either vanishes on the whole 
  graded set $S\cap S_*= (S(n)\cap S_*(n))$ or none of $S\cap S_*$.
\end{lemma}

\begin{proof}
  From Lemma \ref{lem:yesdomin-warmup} (twice),
\[
  I(X,v)=I(S)=I(Y,w).
\]
\end{proof}

  This section closes with the following observation.
  A graded subset $Z=(Z(n))$ of the graded set $\hbd_p$
  {\bf respects simultaneous unitary conjugation}, if 
  for each $n$, $(X,v) \in Z(n)$ and $n\times n$  unitary $U$,
\begin{equation}
  \label{unitary-conj}
        U^{\star}(X,v)U := ((U^{\star}X_1U,\dots,U^{\star}X_gU),U^{\star}v) \in Z(n).
\end{equation}

\begin{lemma}
 \label{lem:variety-of-ideal}
   If $I$ is a subset of  $\cP_d^\delta$, then the
   graded set  $\cZ(I)=(\cZ(I)(n))$ defined by
 \[
   \cZ(I)(n) =\{(X,v) \in\hbd_p(n): f(X)v=0 \mbox{ for all } f\in I\}
 \]
  respects both direct sums and unitary conjugations. 

  Further, if $I\subset J\subset \cP_d^\delta,$ then
  $\cZ(I)(n)\supset \cZ(J)(n)$ for every $n$;  that is, 
  $\cZ(I)\supset \cZ(J)$.   
\end{lemma}

\begin{proof}
  The first statement is evident if $I$ contains a single $q\in\cP_d^\delta$.
  The general result follows by observing 
  that the properties of respecting direct sums and unitary conjugations
  are preserved under (termwise) intersection of graded sets.

  The statement about inclusions is readily verified. 
\end{proof}

\section{Closure with Respect to a Subspace of Polynomials}
 \label{sec:closure}
 In this section a canonical closure operation
 on graded subsets $W=(W(n))$ of the
  graded set $\hbd_p$ is introduced and its 
  properties developed.   
  Recall that the positive integers 
  $d$, $\delta$, and $g$ have all been fixed
 (by $p$) and that $\cP_d^\delta$ denotes the
  $1\times \delta$ matrices whose entries are
  free polynomials of degree at most $d$ in $g$
  freely non-commuting symmetric variables.

 The
 {\bf \pnclos }  of a non-empty graded subset 
  $W=(W(n))$ of $\hbd_p$ 
 is \index{closure} the graded
 set $W_z = (W_z(n))$ where, 
$$
  W_z(n):=\{ (X,v) \in \hbd_p(n) : f(X)v=0 \mbox{ for every }  f \in \cI(W) \}.
$$
 In particular, to say
 $W$ is \ {\bf $\cPNsigd$-closed} means $W_z=W$.
 
\begin{lemma}
 \label{in-closure}
   Let $W=(W(n))$ denote a non-empty graded subset of $\hbd_p$. 
 \begin{enumerate}[(i)]
  \item  In the notation of Lemma \ref{lem:variety-of-ideal},
    $W_z=\cZ(\cI(W))$; 
  \item   if $(X,v)\in\hbd_p(n),$ then $(X,v)\in W_z(n)$ if and only
   if $\cI(X,v)\supset \cI(W);$
  \item   $\cI(W)=\cI(W_z);$ and
  \item
    \label{it:wz-largest}
    if $U=(U(n))$ 
  is a graded subset of $\hbd_p$
  and $\cI(U)=\cI(W)$, then $U\subset W_z$;  that is, $U(n)\subset W_z(n)$
    for every $n$. 
 \end{enumerate}
\end{lemma}

  Note that item \eqref{it:wz-largest} says that $W_z$ is the
  largest graded subset of $\hbd_p$ such that $I(W_z)=I(W)$.

\begin{proof}
  The first item is evident. To prove item $(ii)$, 
  suppose $(X,v)\in W_z(n)$.
  If $q\in \cI(W)$, then $q(X)v=0$ and hence $q\in \cI(X,v)$.
  Thus, $\cI(W)\subset \cI(X,v)$.  Conversely, suppose
  $(X,v)\in\hbd_p(n)$ and 
  $\cI(X,v)\supset \cI(W)$. If $q\in\cI(W)$, then $q\in\cI(X,v)$
  and hence $q(X)v=0$. Hence $(X,v)\in W_z(n)$. 

  Since $(X,v)\in W_z$ implies $\cI(X,v)\supset \cI(W)$, it follows
  that $\cI(W_z)\supset \cI(W)$.  On the other hand,
  since $W\subset W_z$, the inclusion $\cI(W)\supset \cI(W_z)$
  and the equality $\cI(W)=\cI(W_z)$ follows.

  Finally, suppose $\cI(U)=\cI(W)$ and let $(X,v)\in U$ be
  given. If $q\in \cI(W)$, then $q\in\cI(U)$ and hence
  $q(X)v=0$. Thus, $(X,v)\in W_z$ and hence $U\subset W_z$.
\end{proof}

  The following lemma collects basic facts
  about the $\cPNsigd$-closure operation.
  The statement and proof extensively use 
  the following conventions.  Given 
  graded subsets $A=(A(n))$ and $B=(B(n))$
  of the graded set $\hbd_p$, the notation
  $A\subset B$ means $A(n)\subset B(n)$ 
  for each $n$.  Similarly, the notation
  $A\subsetneq B$ means $A\subset B$ and
  there is an $m$ so that $A(m)\subsetneq B(m)$. 

\begin{lemma}
\label{lem:clos}
 Suppose $\hbd_p \supset A, B$ are non-empty graded sets 
 which respect direct sums. 
\begin{enumerate}[(i)]
\item
  $A\subset A_z$;
\item
 If $A \supset B$, then $\cI(A) \subset \cI(B);$
\item
If $\cI(A) \subset \cI(B)$, then
   $A_z   \supset B_z  \supset B;$
\item
  If $B\subset A$, then $B_z\subset A_z$;
\item
 If $B$ is $\cPNsigd$-closed and $B\subsetneq A$, then
 $\cI(A)\subsetneq \cI(B)$;
\item
If $A_1 \supset A_2 \supset \cdots $
is a decreasing
sequence of non-empty \pnclosed \ sets,
then there is an $m$ such that $A_m=A_\ell$ for all $\ell\ge m$; and 
\item
\label{it:minelement}
 A non-empty collection $\mathfrak{T}$
 of non-empty $\cPNsigd$-closed subsets
 of $\hbd_p$ contains a minimal element; i.e., there
 exists a set $T\in\mathfrak{T}$ such that if $A\subset T$
 and $A\in\mathfrak{T}$, then $A=T$.
\end{enumerate}
\end{lemma}

\begin{proof}
  The first four items are obvious.

 To prove $(v)$, note that by $(ii)$, $\cI(A)\subset \cI(B)$. On
 the other hand,
 if $\cI(A)=\cI(B)$, then by $(iii)$, $A_z = B_z$. But then,
  because $B$ is $\cP_d^\delta$ closed, 
\[
  B_z=B \subsetneq A \subset A_z =  B_z,
\]
 a contradiction.

 Item $(vi)$  holds because, by $(v)$,  $\cI(A_1)
 \subset \cI( A_2) \subset \cdots $
 is an increasing nest of subspaces of
 the finite dimensional vector space $\cPNsigd$.
 Thus there is an $m$ such that
 $\cI(A_\ell)=\cI(A_m)$ for all $\ell \ge m$.
 Using $(iii)$ twice and the fact that each $A_\ell$
 is $\cPNsigd$-closed, it follows that $A_\ell=A_m$
 for $\ell\ge m$.

 To prove (\ref{it:minelement}),
 choose $A_1\in\mathfrak{T}$. If $A_1$ is not minimal,
 then there exists $A_2\in\mathfrak{T}$ such that
  $A_1 \supsetneq  A_2.$
 Continuing in this fashion,  eventually produces a minimal set $T$
 as the alternative is a
 nested strictly decreasing sequence
\[
  A_1\supsetneq A_2 \supsetneq A_3 \supsetneq \dots
\]
 from $\mathfrak{T}$ which contradicts $(vi)$.
\end{proof}

  Facts about the relation between dominating
  points   and $\cPNsig$-closures
  are collected in the next lemma.
  Recall the characterization of
  dominating points given 
  in Lemma \ref{lem:yesdomin-warmup}.
  If $A$ is
  a graded subset of $\hbd_p$, let 
  $A_*$ denote the graded set $A_*=(A_*(n))$. 
  Recall, if $B$ is also a graded set,
  then $A\cap B$ is the graded set
  $(A(n)\cap B(n))$.

\begin{lemma}
 \label{lem:domin}
    Suppose $\hbd_p \supset A, B$ are non-empty
  graded sets which respect direct sums. 
\begin{enumerate}[(i)]
 \item 
  If $ A \supset B,$ then  $A_* \subset B_*$;
 \item 
   \label{it:Astvsz}
     $A_* = (A_z)_*$;
 \item 
   \label{it:BBst}
    $B\cap B_*$ is non-empty;
 \item 
  \label{it:domin-4}
  \begin{equation}
       B\cap B_* \subset \{(X,v)\in\hbd_p : \cI(X,v)=\cI(B)\} \mbox{ and;}
  \end{equation}
 \item 
   If $A$ is $\cPNsigd$ closed, then
  \[
   A\cap A_* =\{(X,v)\in \hbd_p : \cI(X,v)=\cI(A) \}.
  \]
    Hence for any $B$,
  \[
   B_z \cap B_* =\{(X,v)\in\hbd_p: \cI(X,v)=\cI(B)\}.
  \]
\end{enumerate}
\end{lemma}

\begin{remark}
 \label{BcapB-star}
   Note that item $(iii)$ 
   is Lemma \ref{lem:yesdomin}
   and $(iv)$ is part of 
    Lemma \ref{lem:yesdomin-warmup}.
\qed \end{remark}

\begin{proof}
   To prove item $(i)$ observe, 
  if $(X,v) \in A_*(n)$, then,
  by Lemma \ref{lem:yesdomin-warmup}
  and Lemma \ref{lem:clos}$(ii)$,
   $\cI(X,v) \subset \cI(A)  \subset \cI(B)$.
  Thus, by another application of Lemma \ref{lem:yesdomin-warmup},
  $(X,v) \in B_*(n)$.

   By Lemma \ref{lem:clos}$(i)$,  $A\subset A_z$.
   Thus, by part $(i)$ of this lemma, $A_*\supset (A_z)_*$.
   On the other hand, if $(X,v)\in A_*(n)$, then,
   in view of Lemma \ref{in-closure}$(iii)$, 
  \[
   \cI(X,v)\subset \cI(A) = \cI(A_z)
  \]
   and thus $(X,v)\in (A_z)_*(n)$. Hence $A_*\subset (A_z)_*$
   and item $(ii)$ is proved.

   It remains to prove item $(v)$.
   One inclusion follows from $(iv)$. 
     To prove the other inclusion, suppose $A$ is $\cPNsigd$-closed,
   $(X,v)\in\hbd_p$, and $\cI(X,v)=\cI(A)$.
   Since
   $\cI(X,v) \supset \cI(A)$ and $A$ is $\cPNsigd$-closed,
   item $(ii)$ of  Lemma \ref{in-closure} implies
   $(X,v)\in A$. On the other hand, $(X,v)\in A_*$
   since $\cI(X,v) \subset \cI(A)$. Thus the reverse
   inclusion holds and the proof of the first part of
   $(v)$ is  complete. The second part of $(v)$ 
   follows from the first part and item $(iii)$ of Lemma \ref{in-closure}.
\end{proof}

For a monic  linear pencil $L,$
let $i(L)$ denote the graded subset $(i(L)(n))$
  of the graded set $\hbd_p$ defined by 
\[
  i(L)(n):= \ \{(Y,w)\in\hbd_p(n) : \ L(Y) \mbox{ is  invertible}\}.
\]
  \index{$i(L)$}
  If $S$ is a graded subset of $\hbd_p$, then $L$ is
  said to be {\bf singular on $S$} if $L(X)$ is not
  invertible for each $n$ and $(X,v)\in S(n)$; i.e., if
  $S(n)\cap i(L)(n)$ is empty for each $n$.

\begin{proposition}
 \label{propL4}
  Let $S=(S(n))$ be a non-empty
  graded subset of the graded set $\hbd_p.$ Suppose
  $S$  respects direct
  sums and  $L$ is a monic  linear pencil. If
 \begin{enumerate}[(i)]
  \item  $L$ is singular on $S_*$; and
  \item  $\emptyset \ne i(L) \subset S$,
 \end{enumerate}
 then $i(L)_z$ is properly contained in $S_z$; i.e., there is
  an $m$ such that 
 \[
   i(L)_z(m) \subsetneq S_z(m). 
 \]
\end{proposition}

\begin{proof}
    Item $(ii)$ and Lemma \ref{lem:clos}$(iv)$
     imply  $i(L)_z\subset S_z$.
   Arguing by contradiction,
  suppose that   $i(L)_z=S_z$.  Then, from Lemma
  \ref{lem:domin}$(ii)$ (twice),
 \[
     i(L) \cap i(L)_* = i(L) \cap (i(L)_z)_*
       = i(L)\cap (S_z)_* =i(L)\cap S_*.
 \]
  On the other hand, since $i(L)=(i(L)(n))$ 
  is a non-empty graded subset of 
   $\hbd_p$ which \rds, by Lemma \ref{lem:yesdomin},
  there is an $m$ and an $(X,v)\in i(L)(m)\cap i(L)_*(m).$
  Hence there is an $(X,v)\in i(L)(m)\cap S_*(m)$.
  But then $L(X)$ is invertible, 
   since
  $(X,v) \in i(L)$ and on the other hand, by $(i)$,
  $L(X)$ is singular because $(X,v) \in S_*(m)$.
  This contradiction proves the indicated inclusion is proper.
\end{proof}

\section{Convexity and the Invertibility Set}
 \label{sec:sepPencils}
  This section contains proofs of two facts about
  the convex graded set $\cD_p.$ 
   First, it is in fact an open matrix convex set
  (see Definition \ref{def:matrix-convex} below); 
  and second, 
   membership in $\cD_p$ and its boundary is
  determined by compressions to subspaces of dimension
  at most $\nu=\delta\sum_{0}^d g^j.$ (Recall,
  $p$ is $\dd\times \dd$ matrix-valued,
  $d$ is the degree of $p$, and $g$ is the number of variables.)

\subsection{Matrix Convexity}
  A graded subset $S=(S(n))$ of $\smatg$ 
  \df{respects simultaneous unitary conjugation}
  if
  for each $n$, $X\in S(n)$ and each $n\times n$ unitary matrix,
\[
  U^{\star}XU = (U^{\star}X_1U,\dots,U^{\star}X_gU)\in S(n).
\]
  This is analogous to \eqref{unitary-conj}.
  The following lemma applies to any 
  $\cD_q$, whether convex or not.  The second item 
  has already been used repeatedly. 

\begin{lemma}
 \label{lem:convexities}
  Suppose $q\in\cP^{r\times r}$ is symmetric and $q(0)$ is invertible.
 \begin{enumerate}[(i)]
   \item   The graded set 
     $\cD_q$ respects simultaneous unitary conjugation; and 
   \item 
       $\cD_q$ respects direct sums.
 \end{enumerate}
\end{lemma}

\begin{proof}
  The first item follows from the fact that  $q(U^{\star}XU)=U^{\star}q(X)U$
  and the second from $q(X\oplus Y) = q(X)\oplus q(Y)$. 
\end{proof}

   Recall, by definition,
   $\cD_p=(\cD_p(n))$ is convex if each $\cD_p(n)$ is convex.

\begin{lemma}
  \label{lem:restrict}
    If $\cD_p$ is convex, $X\in\smatng$, $Y\in\mathbb S_m(\mathbb R^g)$,
     and $X\oplus Y\in\cD_p(n+m)$, then $X\in\cD_p(n)$ and $Y\in\cD_p(m)$.
\end{lemma}

\begin{proof}
   Let $Z=X\oplus Y\in\cD_p(n+m)$. 
   By convexity, $tZ\in\cD_p(n+m)$ for $0\le t\le 1$.
   It follows that  $p(tX)$ is invertible
   for $0\le t \le 1$ and so there is a path from $0$ to $X$ lying
   in $\cD_p(n)$. Thus $X\in\cD_p(n)$. Likewise for $Y$.
\end{proof}

\begin{remark} \label{rem:billWhines}
\rm
 It is not clear if Lemma \ref{lem:restrict}
 remains true with the weaker
 hypothesis
 that the closure of $\cD_p$ is convex. 
\qed \end{remark}

\begin{definition}\rm
 \label{def:matrix-convex}
  For the present purposes a graded set  $\mathcal C=(\mathcal C(n)),$
  where each $\mathcal C(n) \subset \mathbb S_n(\mathbb R^g),$
  is a bounded open {\bf matrix convex} \index{matrix convex} set if
\begin{enumerate}[(i)]
   \item each $\cC(m)$ is open and contains
     $0=(0,\dots,0)\in\mathbb S_m(\mathbb R^g)$;
   \item  { $\cC$  respects direct sums};
   \item {\bf $\cC$  respects
   simultaneous conjugation with contractions}:
   if $Y\in\cC(m)$ and $F$ is an $m\times k$ contraction,
    then
  \[
      F^{\star} Y F= (F^{\star}Y_1F,\dots,F^{\star} Y_gF)\in \cC(k); \mbox{ and}
  \]
   \item  each $\cC(m)$ is convex and bounded.
  \end{enumerate}
\end{definition}

   There are some harmless redundancies in the conditions
   above.  
    It is easy to see  that the convexity of  $\cC(m)$ 
     actually follows from items (ii)
    and (iii).   Indeed, given $X,Y\in\mathcal C(n)$,
   choose $F$ to be the $2n\times n$ matrix
 \[
   F= \frac{1}{\sqrt{2}} \begin{pmatrix} I_n \\ I_n\end{pmatrix}
 \]
   and note that
 \[
    \frac{X_j+Y_j}{2} = F^{\star} \begin{pmatrix} X_j & 0 \\ 0 & Y_j
    \end{pmatrix} F \qquad \mbox{for each } j.
 \]
   Similarly, if it assumed that $\cC$ is not empty, then that 
   $0\in\cC(n)$ for all $n$ follows from $(iii)$ by choosing $F=0$.

   An immediate consequence of item (iii) is,
    if $X\in\mathbb S_n(\mathbb R^g)$,
     $Y\in\mathbb S_m(\mathbb R^g)$ and $X\oplus Y\in \mathcal C(n+m),$
   then $Y\in \mathcal C(m)$.

\begin{theorem}
 \label{thm:convexities}
   If $p$
   satisfies the conditions of Assumption \ref{assume},
   then $\cD_p$ is a bounded open matrix convex set.
\end{theorem}

\begin{proof}
   That $\cD_p$ is closed with respect to direct sums is
  part of Lemma \ref{lem:convexities} (and does not depend upon
  convexity or boundedness).

  To prove that $\cD_p$
  is closed with respect to simultaneous conjugation
  by contractions, suppose that $X\in\cD_p(n)$ and
  $F$ is a given $n\times k$ contraction.
  Let $U$ denote the Julia matrix (of $F$),
\[
  U=\begin{pmatrix} F  & (I_n-FF^{\star})^\frac12 \\
                 -(I_k-F^{\star}F)^{\frac12} & F^{\star} \end{pmatrix}.
\]
  Routine calculations show $U$ is unitary.

  Let $0$ denote the $g$-tuple of zero matrices of size $k\times k$.
  Then, since $X\in\cD_p(n)$ and $0\in\cD_p(k),$ the direct
  sum $X\oplus 0$ is in $\cD_p(n+k)$. Since  $\cD_p(n+k)$ is closed
  with respect to unitary conjugation both the 
  $g$-tuples of matrices
 \[
  \begin{split}
    Y=&U^{\star} \begin{pmatrix} X & 0\\0 & 0 \end{pmatrix} U \\
    Z=&\begin{pmatrix} I &0 \\0 & -I \end{pmatrix} Y
       \begin{pmatrix} I&0\\0 & -I\end{pmatrix}
  \end{split}
 \]
  are in $\cD_p(n+k)$.
  Using the convexity assumption on $\cD_p(n+k)$,
 \[
   \frac12 (Y+Z)=\begin{pmatrix} F^{\star} X F & 0\\0&
        (I-FF^{\star})^\frac12 X (I-FF^{\star})^\frac12 \end{pmatrix}
 \]
  is in $\cD_p(n+k).$  An application of
 the Lemma \ref{lem:restrict}
 implies $F^{\star}XF\in \cD_p(n)$.

   By hypothesis $\cD_p$ is bounded.
\end{proof}

\subsection{Compressions}
  Recall $\cP^\delta_d$ denotes the $1\times \delta$
  matrices whose entries are free polynomials 
  of degree at most  $d$ in  $g$ freely non-commuting variables. 
  Given $(X,v)\in \smatng \times
  (\mathbb R^\dd \otimes \mathbb R^n)$ define
  the subspace  $\cM=\cM(X,v)$ of  $\RR^{n}$ by
 \begin{equation}
  \label{eq:cM}
  \cM:  = \ \{q(X)v: \ q\in\cPNsigd \} \subset \mathbb R^n.
 \end{equation}  \index{$\cM$}
  Explicitly, 
  $v$ is a column vector of length $\delta$ with entries from
  $\mathbb R^n$ and 
 \begin{equation}
  \label{eq:qXv}
   q(X)v= \begin{pmatrix} q_1(X) \dots q_\dd(X) \end{pmatrix}
     \begin{pmatrix} v_1 \\ \vdots \\ v_\dd \end{pmatrix}
    =\sum q_j(X)v_j,
 \end{equation}
  where each $q_j$ is free  polynomial of degree at most $d$. 

   Let $P_{\mathcal M}$ denote the projection of $\mathbb R^n$
  onto  $\mathcal M$.  Consistent with previous 
  usage,  the notation $P_{\mathcal M}X|_{\cM}$
  is shorthand for  $(P_{\cM} X_1|_{\cM}, \dots, P_{\cM}X_g|_{\cM})$. 
  The integer $\nu= \delta \sum_{j=0}^{\Ndsig} g^j$ \index{$\nu$},
  the dimension 
  of the vector space $\cP_d^\delta$, is an upper bound
  for the dimension of the vector space $\cM$.

\begin{lemma}
 \label{lem1}
   Suppose $p$  satisfies the hypotheses of
   Assumption \ref{assume} and $n$ is a positive integer.
    If $(X,v) \in \hbd_p(n)$  and $\mu$ is the 
  dimension of $\cM=\cM(X,v),$  
   then  $(P_{\cM} X|_\cM, v) \in \hbd_p(\mu).$ In fact, 
   $t \PM X|_{\cM} \in \cD_p(\mu)$
   for $0\le t < 1$ and  $p(P_{\cM}X|_{\cM})v = 0$.
\end{lemma}

\begin{proof}
   From Lemma \ref{inhboundary}, $tX\in\cD_p(n)$
   for $0\le t<1$.
   Let  $V$ denote the inclusion of
   $\cM$ into $\mathbb R^n$.
   Since $V$ is a contraction and,
   by Theorem \ref{thm:convexities}, $\cD_p$
   is a (open) matrix convex set, 
   $t P_{\cM} X|_{\cM} =  V^{\star} tXV
   \in \cD_p(\mu)$.
   
    Writing $v$ as in equation \eqref{eq:qXv}, 
     for any word $w$ of length at most $d$ and 
   any $1\le j\le \delta$, 
 \[
      w(\PM X|_{\cM})v_j = \PM w(X)|_{\cM}v_j  = \PM w(X)v_j.
 \]
   Hence,
 \[
   p(\PM X|_{\cM})v= (I_\delta\otimes \PM) p(X)v = 0.
 \]
\end{proof}

\section{Separating Monic Linear Pencils}
\label{sec:hahnHello}
   This section  develops a 
   refinement of the matricial Hahn-Banach
   separation theorem of Effros-Winkler for
   the graded set $\cD_p$, 
   Proposition \ref{prop:Lindepm} in
   Subsection \ref{subsec:dominate}.
   A version of the Effros-Winkler
   separation theorem is the topic
   of the first subsection.

\subsection{A Version of the  Effros-Winkler Theorem}
 \label{sec:EW-really}
   This subsection
  contains a proof of  the separation theorem
  of Effros and Winkler \cite{EW97} in the
  special case of certain  matrix convex subsets
  of $\mathbb S(\mathbb R^g)=(\mathbb S_n(\mathbb R^g))_{n=1}^\infty$.
  The specialization makes the proof
  of Proposition \ref{prop:sharp}, which
  is applied in the following subsection, 
  simpler than that of the strictly more general version
  in \cite{EW97}.
  On the other hand Proposition \ref{prop:sharp} is
  not explicitly covered by the results in \cite{EW97}.

\def\oX{X^{\mbox{b}}}

\def\be{\mathbf{e}}
\def\LL{\cL}
\def\cTn{\mathcal T_n}
\def\tr{\mbox{tr}}
\def\hh{T}
\def\hhs{\mathfrak{T}}
\def\Rn{\mathbb R^n}

 Given a positive integer $n$, let
 $\cTn$ denote the positive semi-definite
 $n\times n$ matrices (with real entries)
 of trace one.  Each $T \in \cTn$ corresponds
 to a state on $M_n$, the $n\times n$ matrices, via the
 trace,
\[
  M_n \ni A \mapsto \tr(AT).
\]
  Note that $\cTn$ is a convex, compact subset of $\mathbb S_n,$
  the symmetric $n\times n$ matrices.  

 The following
 lemma is a version of  Lemma 5.2 from [EW].
 An affine linear mapping $f:\mathbb S_n \to \mathbb R$ is a
 function of the form $f(x)=a_f +\lambda_f(x)$,
 where $\lambda_f$ is linear and $a_f\in\mathbb R$.

\begin{lemma}
 \label{lem:cone}
  Suppose $\mathcal F$ is a convex set of affine linear
  mappings $f:\mathbb S_n \to \mathbb R$. If for each $f\in \mathcal F$
  there is a $\hh \in\cTn$ such that $f(\hh)\ge 0$,
  then there is a $\hhs\in \cTn$ such that
  $f(\hhs)\ge 0$ for every $f\in\mathcal F$.
\end{lemma}

\begin{proof}
   For $f\in\mathcal F$, let
 \[
   B_f =\{\hh\in \cTn: f(\hh)\ge 0\}\subset \cTn.
 \]
  By hypothesis each $B_f$ is non-empty and
  it suffices to prove that
 $$
   \cap_{f\in\mathcal F} B_f \neq \emptyset.
 $$
  Since each $B_f$ is compact, it suffices to
  prove that the collection $\{B_f: f\in\mathcal F\}$
  has the finite intersection property.  Accordingly,
  let $f_1,\dots,f_m\in\mathcal F$ be given. Arguing
  by contradiction, suppose
\[
  \cap_{j=1}^m B_{f_j} =\emptyset.
\]
Define $F:\mathbb S_n \to \mathbb R^m$  by
\[
  F(\hh)=(f_1(\hh),\dots,f_m(\hh)).
\]
 Then $F(\cTn)$ 
 is both convex and compact because $\cTn$
 is both convex and compact and each
  $f_j$, and hence $F$, is affine linear.
 Moreover, 
  $F(\cTn)$  does not intersect 
 \[
   \mathbb R^m_+=\{x=(x_1,\dots,x_m): x_j\ge 0 \mbox{ for each } j\}.
 \]
 Hence there is a linear functional $\lambda:\mathbb R^m \to \mathbb R$
 such that $\lambda(F(\cTn))<0$ and $\lambda(\mathbb R_+^m) \ge 0$.
 There exists $\lambda_j$ such that
$
 \lambda(x) = \sum \lambda_j x_j.
$
  Since $\lambda(\mathbb R^m_+) \ge 0$ it follows that each
  $\lambda_j\ge 0$ and since $\lambda\ne 0$, for at least one $k$,
  $\lambda_k>0$.  Without loss of generality, 
  it may be assumed that $\sum \lambda_j=1$. 
  Let
\[
  f=\sum \lambda_j f_j.
\]
 Since $\mathcal F$ is convex, 
 it follows that $f\in\mathcal F$. On the other hand,
 $f(T)=\lambda(F(T)).$  Hence 
 if $T\in \cTn,$ then  $f(T)<0$.
 Thus, for this $f$ there does not exist
 a $T\in \cTn$ such that $f(T)\ge 0$,
 a contradiction which completes the proof.
\end{proof}

\begin{lemma}
 \label{lem:cT}
  Let $\cC=(\cC(n))$ denote an open matrix convex subset of
  the graded set $\smatg.$ 
 \ Let $n$ and a linear functional \
  $\Lambda: \smatng\to \mathbb R$ be given.
  If 
   $
    \Lambda(X) \le 1
  $
   for each $X\in\cC(n)$,
  then there is a
  $\hhs\in \cTn$ such that
  for each $m$, each $Y\in \cC(m),$ and each
  $m\times n$ contraction (matrix) $C$, 
 \[
   \Lambda(C^{\star} Y C) \le \tr(C\hhs C^{\star}).
 \]

\end{lemma}


\begin{proof}
 Given a positive integer $m,$ a  tuple $Y$ in $\cC(m)$
 and an $m\times n$ contraction matrix $C$, define
 $f_{Y,C}:\mathbb S_n \to \mathbb R$ by
\[
  f_{Y,C}(\hh)=\tr(C\hh C^{\star}) - \Lambda(C^{\star}YC).
\]

  Now we show that the collection $\mathcal F=\{f_{Y,C}:Y,C\}$ is
  a convex set. Start with a positive integer $s,$ 
  nonnegative numbers $\lambda_1,\dots,\lambda_s$ with
  $\sum \lambda_j=1,$ and with $(Y_j,C_j)$ 
  for $j=1,\dots,s$ where $Y_j\in\cC(m_j)$ and $C_j$ are $m_j\times n$
  contraction matrices. Let $Z=\oplus Y_j$
  and let $F$ denote the (block) column matrix with entries
  $\sqrt{\lambda_j} C_j$. Then $Z\in \cC(m)$ where $m=\sum m_j$ and
\[
  F^TF =\sum \lambda_j C_j^TC_j \preceq \sum \lambda_j I =I.
\]
 By definition
\[
 \sum \lambda_j C_j^T Y_j C_j = F^T ZF
\]
 and
\[
 \sum \lambda_j \tr(C_j T C_j^T) = \tr(F TF^T).
\]
 Therefore
\[
  \sum \lambda_j f_{Y_j,C_j}(\hh) = f_{Z,F}(\hh).
\]

  If $C$ has (operator) norm one, choose $\hh=\gamma \gamma^{\star}$
  where $\gamma$ is a unit vector such that
\[
  \|C\gamma\|=\|C\|=1.
\]
  It follows that $\gamma\gamma^{\star}\in \cTn$ and 
\[
  f_{Y,C}(\gamma \gamma^{\star})
    =\|C\|^2 - \Lambda(C^{\star} Y C) = 1-\Lambda(C^{\star}YC).
\]
 Since $C^{\star}YC \in \cC(n)$, the right hand side above is non-negative.
 If the contraction $C$ does not have norm $1$, but is not zero, a simple scaling
 argument shows that $f_{Y,C}(\gamma \gamma^{\star})\ge 0$ still.
 Consequently, for each $f_{Y,C}$ there is a $T\in\cTn$ such
 that $f_{Y,C}(T)\ge 0$. From 
 Lemma \ref{lem:cone}, there is a $\hhs \in \cTn$ such that
 $f_{Y,C}(\hhs)\ge 0$ for every $Y$ and $C$. 
\end{proof}

  Given $\epsilon>0$,
  the {\bf free $\epsilon$-neighborhood of $0$},
  denoted  $\cN_\epsilon,$ is the graded set
  $(\cN_\epsilon(n))_{n=1}^\infty$ where
 \[
   \cN_\epsilon(n) =\{X\in\smatng : \sum \| X_j \| < \epsilon  \}.
 \]

\begin{lemma}
 \label{lem4}
   If $p$ satisfies the conditions of Assumption \ref{assume},
   then $\cD_p$ contains an $\epsilon>0$ neighborhood of $0$; i.e.,  
   there is an $\epsilon>0$ such that
   $\cN_\epsilon(n)\subset \cD_p(n)$ for each $n$.

   Moreover,
   if the monic linear
   pencil $L=I+\sum A_j x_j$
   is positive definite on $\cD_p$,
   then $\|A_j\|\le \frac{1}{\epsilon}$ for each $j$.
\end{lemma}

\begin{proof}
  Write $p$ as in equation \eqref{pww}. Thus
  each $p_w$ is a $\dd\times \dd$ matrix.
  Let $M$ denote the maximum of $\{\|p_w\| : 1\le |w|\le d\}.$
  Let $\tau=\sum_{1}^d g^j$. Thus $\tau$
  is the number of words $w$ with $1\le |w| \le d$.

  Let $0<\Delta$ denote the minimum of
  $\{|\lambda|: \lambda \mbox{ is an eigenvalue of } p(0)\}$.
  Choose $\epsilon =\min\{1,\frac{\Delta}{\tau(M+1)}\}$.

   Let $X\in\smatng$ be given.
   If $\|X_j\| <\epsilon$ for $1\le j\le g$, then
  $\|w(tX)\|\le \frac{\Delta}{\tau(M+1)}$ for non-empty words $w$
   and $0\le t\le 1$. Hence,
 \[
    \| \sum_{1\le |w|\le d} p_w \otimes w(tX) \|
      \le \sum_{1\le |w|\le d} \|p_w\| \, \|w(tX)\| <\Delta.
 \]
   It follows that $p(tX)$ is invertible
   for $0\le t\le 1$ and thus $X\in\cD_p(n).$
   Consequently $\cD_p(n)$ contains 
   $\cN_\epsilon(n).$

   Now suppose $L$ is a monic  linear pencil which
   is positive definite on $\cD_p$ and thus
    on $\cN_\epsilon$. For $0\le t<\epsilon$,
   the points $\pm t e_j$ are in $\cN_\epsilon(1)$
   and hence
    $L(\pm t e_j) = I\pm t A_j \succeq 0$. It follows
  that $\pm A_j \preceq \frac{1}{\epsilon}I$ and thus
  $\|A_j\|\le \frac{1}{\epsilon}$.
\end{proof}

\begin{proposition}
 \label{prop:sharp}
 Let $\cC=(\cC(n))$ denote a bounded  open    matrix convex
 subset of the graded set  $\smatg$ which
 contains a free $\epsilon$-neighborhood of $0.$
 If
 {\rm $\oX\in\smatng$} is in the boundary of $\cC(n)$,
 then there is a monic   linear pencil
 $L$ (of size $n$) such that $L(Y)\succ 0$ for all $m$ and
 $Y\in\cC(m)$ and such that
 {\rm $L(\oX)$} is singular.
\end{proposition}

\def\beq{\begin{equation}
}
\def\eeq{\end{equation}}

\begin{proof}
 By the usual Hahn-Banach separation theorem
 and the assumption that $\cC(n)$ contains
 an $\epsilon$-neighborhood of $0$,
 there is a linear functional $\Lambda:\smatng \to \mathbb R$
 such that $\Lambda(\oX) =  1 > \Lambda(\cC(n))$.

 $\mbox{From}$ Lemma \ref{lem:cT} there is a positive
 semi-definite $n\times n$ matrix $T$ of trace one such that
\beq
\label{eq:absLMI}
\tr(CTC^{\star})- \Lambda(C^{\star}YC) \ge 0
\eeq
 for each $m$, each $m\times n$ contraction $C$, and
 each $Y\in\cC(m)$.
 Note this inequality is sharp in the sense,
\beq
\label{eq:absLMI0}
\tr(T)- \Lambda(\oX) = 0.
\eeq
The rest of the proof amounts to
expressing  \eqref{eq:absLMI} in a concrete way
in terms of a monic  linear pencil.

 Let $\{\be_1,\dots,\be_g\}$ 
  denote the
 standard orthonormal basis for $\mathbb R^g.$
 Given $1\le \ell \le g$, define
 a bilinear form on $\Rn$ by
\[
  \cB_\ell(c,d)=\frac12 \Lambda((cd^{\star} + dc^{\star})\otimes \be_\ell)
\]
   for  $c,d\in\mathbb R^n$.
  There is a unique real symmetric $n\times n$ matrix $B_\ell$
  such that
\[
  \cB_\ell(c,d)=\langle B_\ell c, d\rangle.
\]

 Let $L_B$ denote the linear polynomial  $L_B(x)=\sum_1^g B_j x_j$.
 Fix a positive integer $m$ and let 
 $\{e_1,\dots,e_m\}$ denote the standard orthonormal basis for
  $\mathbb R^m$. 
 Let $Y=(Y_1,\dots,Y_g)\in\cC(m)$ be given and consider
  $L_B(Y)$.
 Given a  vector
 $\gamma=\sum_{j=1}^m  \gamma_j \otimes e_j$
 contained in $\Rn\otimes \mathbb R^m$, compute
\[
 \begin{split}
 \langle L_B(Y)\gamma,\gamma \rangle
  =& \sum_{i,j}\sum_\ell \langle  B_\ell\gamma_j,
           \gamma_i\rangle \langle Y_\ell e_j,e_i\rangle \\
  = & \frac12 \sum_{i,j} \sum_\ell \Lambda( (\gamma_j \gamma_i^{\star}
              +\gamma_i \gamma_j^{\star})\otimes \be_\ell)
                \langle Y_\ell e_j,e_i\rangle  \\
  = & \Lambda (\sum_{i,j} \gamma_i (\sum_\ell \langle Y_\ell e_j,e_i\rangle
            \otimes \be_\ell) \gamma_j^{\star}) \\
  = & \Lambda(\Gamma  Y \Gamma^{\star}),
 \end{split}
\]
 where $\Gamma$ is the matrix with $j$-th column $\gamma_j$.
 Using equation \eqref{eq:absLMI}
\[
 \begin{split}
  \Lambda(\Gamma  Y \Gamma^{\star}) \le & \tr(\Gamma^{\star} T \Gamma) \\
  = & \sum \langle T \gamma_j,\gamma_j \rangle \\
   =& \sum \langle (T\otimes I) \sum_j \gamma_j\otimes e_j,
       \sum_k \gamma_k \otimes e_k \rangle \\
   =& \langle (T\otimes I) \gamma,\gamma\rangle.
 \end{split}
\]
 Thus,  the  linear pencil
$T-L_B$ defined by
 $(T-L_B)(x)=T-\sum B_j x_j$
satisfies
\beq
\label{eq:Tlmi}
 [T - L_B](Y) \succeq 0
 \eeq
 for every $m$ and $Y\in\cC(m)$.

 Since $\cC$ contains the $\epsilon$-neighborhood of $0$,
  it contains $\pm \frac{\epsilon}{2} e_j\in \mathbb R^g.$
  Hence, 
 \[
    0\preceq T - \pm \frac{\epsilon}{2} L_B(e_j) 
    = T- \pm \frac{\epsilon}{2} B_j.
 \] 
 Thus,  while $T$ need not be invertible, it does satisfy
  $-T\preceq \frac{\epsilon}{2}  B_j \preceq T$ 
   for  each $j$ and hence restricting to 
  the range of $T$ (kernel of $T^{\star}$)
  it can be assumed (passing to a space of smaller dimension
  if necessary), that $T$ is invertible.  Finally,
  multiplying left and right by $T^{-\frac12}$ produces a
  linear polynomial  $\cL(x) = \sum_j A_j x_j$
  such that $(I- \cL )(Y)\succeq 0$ if
  and only if $(T-L_B)(Y) \succeq 0$.

 On the other hand,  computing as above, \eqref{eq:absLMI0}
   becomes
 $$
  \langle (T-L_B)(\oX) e,e\rangle =0 \ \ \
 \mbox{with}  \ \ \ e=\sum e_j\otimes e_j.
$$
 Since $\oX$ is in the closure of $\cC(n)$, 
  $(T-L_B)(\oX)\succeq 0$. 
  Thus $(T-L_B)(\oX) e=0$
  and since $[T \otimes I] e \ne 0,$ it follows
  that $(I- \cL)(\oX)$ is singular.
  Set $L= I - \cL$.

 Finally, the assumption
 that $\cC$ is open implies that $L$ is in fact
 positive definite, not just positive semi-definite, on $\cC$. 
 The proof of this statement is very similar to
 that of Lemma \ref{lem:nochangestart}. The details are omitted.
\end{proof}

\subsection{Effros-Winkler and Invertibility Sets}
 \label{subsec:EW}
  The following lemma is both a  refinement
  and specialization of the free
  Hahn-Banach separation theorem of Effros and Winkler \cite{EW97}.
  It is specialized
  to convex bounded 
  sets $\cD_p=(\cD_p(n));$ and
  refined in that it separates a point on the boundary of $\cD_p(m)$
  from  $\cD_p$.

\begin{lemma}
 \label{sharp-EW}
    Suppose $p$ satisfies the conditions of Assumption \ref{assume}.
    If $X\in\bd_p(m)$, then there exists a monic  linear pencil
   $L$ of size $m$ such that $L$ is positive definite
   on each $\cD_p(n)$ and $L(X)$ is singular.
\end{lemma}

\begin{proof}
  By Theorem \ref{thm:convexities}, $\cD_p$ is
  a bounded open matrix convex set. 
  By Lemma \ref{lem4},  $\cD_p$ 
  contains a free $\epsilon$-neighborhood of $0$. 
  Hence an application of Proposition
  \ref{prop:sharp} proves the lemma. 
\end{proof}

 The following is  a more quantitative version of
 Lemma \ref{sharp-EW}. 
 Recall $\nu=\delta \sum_{0}^d g^j.$

\begin{lemma}
 \label{cor-sharp-EW}
   Suppose $p$ satisfies Assumption \ref{assume}.
   If $(X,v)\in\hbd_p(m)$, then there exists a monic  linear
    pencil $L$ of  size $\ell\le \nu,$ where $\ell$
    is the dimension of 
 \[
  \cM =\cM(X,v)=\{q(X)v: q\in \cPNsigd \}\subset \mathbb R^m,
 \]
   and a non-zero vector $w\in \RR^\ell\otimes \cM$
   such that $L$ is positive definite on each 
   $\cD_p(n)$ and $L(X)w=0$.  
\end{lemma}

\begin{remark}\rm
 \label{form-w}
   In terms of $\{e_1,\dots,e_\ell\},$ the standard basis for $\RR^\ell$,
   there exists $m_1,\dots,m_\ell\in\cM$ such that
   $w=\sum e_\alpha\otimes m_\alpha$.  From the definition of $\cM$, there
   thus exists 
   $q^\alpha \in \cPNsigd$ such that $m_\alpha=q^\alpha(X)v$ and hence,
 \[
   w=\sum e_\alpha \otimes q^\alpha(X)v.
 \]
\qed \end{remark}

\begin{proof}
  Let $Y=\PM X|_{\cM}.$ 
  By Lemma \ref{lem1},  $(Y,v)\in \hbd_p(\ell)$. By Lemma \ref{sharp-EW},
  there exists  a monic  linear pencil $L$
  of size $\ell$ such that $L$ is  positive definite on each 
   $\cD_p(n)$ and $L(Y)$ is singular. Hence, there is
  a non-zero $w\in\RR^\ell\otimes \cM$ such that
  $L(Y)w=0$. Since 
 \[
  \begin{split}
   \langle L(X)w,w\rangle
      = & \langle (I_\ell\otimes \PM) \; L(X) \; (I_\ell\otimes \PM) w,w\rangle \\
     = & \langle  L(Y)w,w\rangle \\
         =& 0, 
  \end{split}
 \]
    and since $L(X)\succeq 0$,  the conclusion $L(X)w=0$
   follows.
\end{proof}

\subsection{Dominating Points and Separation}
 \label{subsec:dominate}
  Proposition \ref{prop:Lindepm} below relates dominating
  points to the separating monic  linear
  pencils  produced by
  Lemma \ref{cor-sharp-EW}. It is the
  main result of this section
  and the last ingredient needed for
  the proof of Theorem \ref{thm:main} in the next section. 

  Let $|w|$ denote the length of a word $w$. By convention,
  $|\emptyset|=0$.

\begin{proposition}
 \label{prop:Lindepm}
   Suppose $p$ satisfies Assumption \ref{assume}.
 If $S=(S(n))$ is a non-empty graded subset
  of the graded set $\hbd_p$ which \rds, 
  then there exists
 a monic  linear pencil $L$ which is positive
 definite on each $\cD_p(n)$ and singular on 
  $S\cap S_*=(S(n)\cap S_*(n));$  that is,
 if $X\in\cD_p(n)$, then $L(X)\succ 0$, and  
 if $(X,v)\in S(n)\cap S_*(n)$, then $L(X)$ is singular. 
 Further, the size of $L$ can be chosen to be at most
 the maximum of the dimensions of
 the subspaces  $\{q(Y)w: q\in\cPNsigd\}$
 over $(Y,w)\in S$ and is therefore at most $\mbox{dim }\cPNsigd=\nu$.
\end{proposition}

\begin{proof}
    Let $\mu$ denote
  the maximum of the dimensions of 
   the subspaces   $\{q(Y)w:q\in\cPNsigd\}$
   for $(Y,w)\in S$.

     Given $(X,v) \in S(m)$,  let $\LC_{X}$
     denote the set of monic  linear pencils
  $L$ of size $\mu$ which are both positive definite
  on each $\cD_p(n)$ and for which $L(X)$ is singular.
  By identifying $L=I+\sum A_j x_j$ with
  the tuple $A=(A_1,\dots,A_g)\in\mathbb S_\mu(\mathbb R^g)$,
  the collection $\Lambda_X$ may be viewed
   as a subset of a finite dimensional vector space.

   By Lemma \ref{cor-sharp-EW},
   each $\LC_{X}$ is
   non-empty.  By Lemma \ref{lem4}
   each $\Lambda_X$ is  bounded.
   If a sequence from $\Lambda_X$ converges
   to the monic  linear pencil $L$,
   then $L(Y)\succeq 0$ for 
   each $n$ and $Y\in \cD_p(n)$.
   By an application of Lemma \ref{lem:nochangestart},
   it follows that $L$ is in fact positive definite
   on each $\cD_p(n)$. Hence $\Lambda_X$ is
   closed and thus compact.

  Given an $s$ and
  $(X^j,v^j)\in 
  \SS(m_j) \cap \SS_*(m_j) \subset \hbd_p(m_j)$
   for $1\le j\le s$, 
  let $(W,u)=\oplus (X^j,v^j).$ Since $S$ is closed
  with respect to direct sums, $(W,u)\in S(m),$
  where $m=\sum m_j$. 

  Concordant with earlier usage, let 
\[
 \cM(W,u):= \  \{ q(W)u: \ q \in \cPNsigd \; \}.
\]
  By   Lemma \ref{cor-sharp-EW} there is
  a monic  linear pencil
  $L=I+\sum A_j x_j$ of size $\mu$  such that
  $L$ is positive definite on each $\cD_p(n)$
  and  a non-zero vector $\gamma\in \RR^\mu \otimes \cM(W,u)$
  such that $L(W)\gamma=0$.
  From the definitions of $\cM(W,u)$ and $\RR^\mu \otimes \cM(W,u)$,
  there exists $q^\alpha\in\cPNsigd$ for
  $1\le \alpha\le \mu,$ such that
\[
 \gamma = \sum_{\alpha=1}^\mu e_\alpha \otimes q^\alpha(W)u.
\]

  Let
\[
  q=\sum_{\alpha=1}^\mu e_\alpha \otimes q^\alpha
     =\begin{pmatrix} q^1 \\ \vdots \\q^\mu \end{pmatrix}.
\]
  Thus $q$ is a $\mu\times \dd$ matrix of
  polynomials of degree at most $d$;  that is,
  $q\in \cP_d^{\mu\times \dd}.$ Further,
\[
 \gamma = q(W)u.
\]

  Up to unitary equivalence (the canonical shuffle),
\[
  L(W)\gamma = L(W)q(W)u =
   \begin{pmatrix} L(X^1)q(X^1)v^1 \\ \vdots \\ L(X^s)q(X^s)v^\mu
    \end{pmatrix}.
\]
  Let
\[
  \gamma_j = q(X^j)v^j
   = \begin{pmatrix} q^1(X^j)v^j \\ q^2(X^j)v^j \\ \vdots \\ q^\mu(X^j)v^j
       \end{pmatrix}.
\]
 Since $L(W)\gamma=0$,
\begin{equation}
 \label{eq:qXqY}
    L(X^j)\gamma_j = 0
\end{equation}
 for each $1\le j\le s.$

  To prove that each $\gamma_j\ne 0$ we
  now invoke the hypothesis that
  each $(X^j,v^j) \in S(m_j) \cap S_*(m_j).$
  If $\gamma_k=0$ (for some $k$), then
  $q^\alpha(X^k)v^k=0$
  for each $\alpha$.
  By Lemma \ref{all-or-none},
  for a fixed  $\alpha$, either $q^\alpha(X^j)v^j=0$
  for every $j$ or $q^\alpha(X^j)v^j\ne 0$ for every $j$.
  Since $q^\alpha(X^k)v^k=0$ it follows that 
  that $q^\alpha(X^j)v^j=0$ for every $j$ and
  every $\alpha$. Thus each $\gamma_j=0$
  and hence $\gamma=0$, a contradiction.

   Since, for each $j$, we have $\gamma_j\ne 0$, but $L(X^j)\gamma_j=0$,
   it follows that
   $L\in \LC_{X^j}$.
   This proves
 \[
   \cap_{j=1}^s \LC_{X^j}
   \ne \emptyset.
 \]
  Consequently, the collection 
  $\{ 
  \LC_{X}
  : (X, v) \in S(n)\cap S_*(n), \ \  1\le n \}$
  of compact sets
  has the finite intersection property. Hence
  the full intersection is non-empty and any
  $L$ in this intersection is positive definite
  on $\cD_p$ and singular on all of $S(n)\cap S_*(n)$ for each $n$
  (meaning, for each $n$,
    if $(X,v)\in S(n)\cap S_*(n)$, then $L(X)$ is singular).
\end{proof}

\begin{corollary}
 \label{cor:Lind}
   If $p$ satisfies Assumption \ref{assume}, then 
   the graded set  $(\hbd_p)_*=(\hbd_p(n)_*)$ is non-empty and
   there is a monic  linear pencil $L$ which is
   positive definite on $\cD_p$ and singular
   on  $(\hbd_p)_*$;  that is, for each $n$, 
  if $X\in\cD_p(n)$ then $L(X)\succ 0$, and if $(X,v)\in (\hbd_p)_*(n)$,
  then $L(X)$ is singular. 
\end{corollary}

\begin{proof}
  Note $\hbd_p \cap (\hbd_p)_* =(\hbd_p)_*$ and apply 
  Proposition \ref{prop:Lindepm} with $S=\hbd_p.$
 \end{proof}

\section{Theorem \ref{thm:main}}
 \label{sec:proof}
  Theorem \ref{thm:main} is an immediate consequence of 
  the following result. 

\begin{theorem}
  \label{thmL}
   Given $p$ satisfying
   Assumption \ref{assume},
   there exists a monic  linear pencil
   $L$ such that $L$ is positive definite on 
  each $\cD_p(n)$
   and $L(X)$ has a kernel for every $n$ and  $X\in\bd_p(n)$.
   Hence, the graded sets $\cD_p=(\cD_p(n))$
   and $\cD_L=(\cD_L(n))=(\{X\in\smatng: L(X)\succ 0\})$
    are equal.
\end{theorem}

\def\ie{that is}

\begin{proof} 
 Recall, for  $L$, a
  monic linear pencil,  $i(L)=(i(L)(n))$ is the graded set defined by 
\[
   i(L)(n):= \ \{(Y,w)\in\hbd_p(n) : \ L(Y) \mbox{ is  invertible } \}.
\]
   We argue by contradiction.
   Accordingly, suppose
   for each monic  linear pencil $L$ which is
   positive definite  on $\cD_p$ the graded
   set $i(L)$ is non-empty.

   Let $\mfS$ denote pairs $(\SS, L)$ with $\SS=(S(n))$ a
   \pnclosed \ graded subset of the graded set $\hbd_p$
     and $L$ a
    monic  linear pencil satisfying:
\begin{enumerate}[(i)]
  \label{Lproperties}
   \item $L$ is positive definite  on $\cD_p$;
   \item    $L$ is singular on  $S_*$;  and
   \item  $i(L) \subset S$.
 \end{enumerate}

  Note that $\mfS$ is not empty since, by Corollary \ref{cor:Lind},
  there is an $L$ such that $(\hbd_p,L)\in\mfS$.
  Let $\mfS_1$ denote the collection of
  graded sets $\SS$ occurring
  in the  pairs $(\SS, L)$ belonging to  $\mfS$.
  Choose a minimal (with respect to term wise set inclusion) 
   graded set $S$ in $\mfS_1,$ whose existence is implied
  by Lemma \ref{lem:clos}\eqref{it:minelement}.
  We will show that $S$ is not minimal, a contradiction
  which will complete the proof.

  Since $S\in\mfS_1$,
  there exists an $L$ satisfying the conditions
  (i)(ii)(iii) with respect to this $S$; \ie,
  $(S,L) \in \mfS$.
  By assumption,  $i(L)(k)\ne \emptyset$ for some $k$. 
  By Proposition \ref{propL4}, $i(L)_z(m)\subsetneq S_z(m)$
  for some $m$.
  Since also $S$ is $\cPNsigd$ closed ($S=S_z$), 
 \begin{equation}
  \label{strict}
    i(L)_z  \subsetneq S.
 \end{equation}

 Using the fact that the graded set $i(L)$
 is non-empty and  respects  direct sums,
 Proposition \ref{prop:Lindepm}
 produces a monic  linear pencil $M$ which
 is positive definite  on each $\cD_p(n)$ and
 singular on each $i(L)(n) \cap i(L)_*(n)$.
 The proof now proceeds by showing
 $(i(L)_z, L\oplus M)\in \mfS$, which, by
 the strict inclusion in equation \eqref{strict},
 contradicts the minimality of $S$.

 From the construction,  $L\oplus M$ is positive
 definite  on each $\cD_p(n)$; \ie, $L\oplus M$
 satisfies condition (i).

 By Lemma \ref{lem:yesdomin} the graded set $i(L)_*$ is not empty.
 Suppose now that $(X,v) \in (i(L)_z)_*(n) = i(L)_*(n)$
  (see Lemma \ref{lem:domin}\eqref{it:Astvsz}).
 If  $(X,v) \in i(L)(n)$, then
 $M(X)$, and hence $(L\oplus M)(X)$ is singular.
 On the other hand, if $(X,v)\notin i(L)(n),$
 then $L(X)$, and hence $(L\oplus M)(X)$ is singular.
 Thus, if $(X,v)\in (i(L)_z)_*$, then $(L\oplus M)(X)$
 is singular. Hence $L\oplus M$ satisfies condition
 (ii) with respect to $i(L)_z$.

 Finally, for each $n$,  $i(L\oplus M)(n) \subset i(L)(n)\subset i(L)_z(n)$
 and thus $L\oplus M$ satisfies condition
 (iii) with respect to $i(L)_z$.
 Hence $(i(L)_z, L\oplus M)\in \mfS$ and the proof is complete.
\end{proof}

\subsection{Estimates on the Size of the Linear Pencil}
 \label{subsec:estimates}
  This subsection gives estimates on the size  of the 
  monic linear pencil $L$ needed in 
    Theorem \ref{thm:main}.
  Recall  $\nu=\dd \sum_0^d g^j$ is
  the dimension of $\cP_d^\delta$.

\begin{lemma}
 \label{prop:quant}
   The size of $L$ need in Theorem \ref{thm:main}
  is at most $\frac{\nu(\nu+1)}{2}$.
\end{lemma}

\begin{proof}[Sketch of proof] 
    The proof of Theorem \ref{thmL} can
  be viewed as a recursive algorithm 
  for constructing $L$ as a
  direct sum $L=\oplus_{j=0}^k L_j$.  The algorithm
  terminates in at most $\nu$ steps
  and, using the estimate afforded by
  Proposition \ref{prop:Lindepm},  
  the dimension of $L_j$ (its matrix size) at the $j$-th step
  is at most $\nu-j$.  Thus $\frac{\nu(\nu+1)}{2}$
  is an upper bound on the size of $L$.
\end{proof}

  In the special case that $p(0)=p_\emptyset$ is positive definite,
   $\cD_p(n)$ is equal to  the component of $0$ of the
   set $\{X\in\smatng : p(X)\mbox{ is positive definite} \}$ and 
   accordingly $\cD_p$ is called the 
  {\bf positivity set} of $p$. \index{positivity set}
  In this case it can be assumed that 
  $p(0)=I_\dd$.  Moreover, the estimate on the size of $L$
  needed in Theorem \ref{thm:main} is reduced dramatically
  from that given in Proposition \ref{prop:quant} above,
  because, as outlined below, the estimate of the size
  of the pencil in Proposition \ref{prop:Lindepm} can be reduced roughly
  by half.

  Let
  $\Nd$ denote the
  largest integer less than or equal to $\frac{d}{2}$.
  Let 
\begin{equation}
 \label{eq:def-nus}
  \nus= \delta \sum_{j=0}^\Nd  g^j.
\end{equation}
  Notice that  $\nus$ is the dimension of the vector
  space $\cPNd$ and, given $(X,v)\in \hbd_p$,
  it is thus an upper bound for  the
  dimension of
$$
  \cMs =\{q(X)v: q\in \cPNd\}.
$$
  The following lemma is a variant of Lemma \ref{lem1},
  using  the smaller space  $\cMs$ instead of
  $\cM$.

\begin{lemma}
 \label{lem1alt}
   Suppose $p \in \cPddd$ satisfies the conditions of Assumption \ref{assume}
   and moreover that $p(0)=I_\dd$.
   If $(X,v) \in \hbd_p(n)$, \ then  $(P_{\cMs} X|_\cMs, v)
    \in \hbd_p(n)$; indeed,
    $t \PMs X|_{\cMs}\in\cD_p(n)$ for $0\le t<1$ and  $p(\PMs X|_{\cMs})v=0$.
\end{lemma}

\begin{proof}
   Just as in Lemma \ref{lem1}, for $0\le t<1$,
   we have    $tP_{\cMs} X|_{\cMs} \in \cD_p$.
  Since $p(0)=I_\dd$, it follows that
  $p(t\PMs X|_{\cMs})\succ 0$ and hence
  $p(\PMs X|_{\cMs})\succeq 0$.

   On the other hand,
   for any word $w$ of length at most $d$,  write
   $w= w_1x_jw_2$ where both words $w_1$ and $w_2$ have length at most
   $\Nd$.
   Write $v\in \mathbb R^\delta \otimes \mathbb R^n$
   as  $v=\sum_{\alpha=1}^\delta e_\alpha\otimes v_\alpha.$
    Since both $w_2(X)v_\alpha$ and $w_1^{\st}(X)v_\beta$ are
   in $\cMs,$ 
 \[
  \begin{split}
   \langle w(P_{\cMs} X|_{\cMs} ) v_\alpha,  \;  v_\beta \rangle
     = & \langle P_{\cMs}  X_j w_2(X)v_\alpha, \; w_1(X)^{\star}
      v_\beta \rangle  \\
     =& \langle X_j w_2(X)v_\alpha, w_1^{\st}(X)v_\beta \rangle \\
     =& \langle w(X) v_\alpha,v_\beta \rangle.
  \end{split}
 \]
  Consequently,
 \begin{equation*}
   \langle p(P_\cMs X|_\cMs )v,  \;  v\rangle
   =\langle p(X)v, \;  v\rangle =0.
 \end{equation*}
  Since also $p(\PMs X|_{\cMs})\succeq 0$,
  it follows that $p(P_\cMs X|_\cMs )v=0.$
\end{proof}

  Applying Lemma \ref{lem1alt} much like in the proof
  of Lemma \ref{cor-sharp-EW} produces the following. 

\begin{lemma}
 \label{cor-sharp-EW-alt}
    Suppose $p$ satisfies Assumption \ref{assume}
   and further that $p(0)=I_\dd$. 
    If $(X,v)\in\hbd_p(n)$, then there exists a monic  linear
    pencil $L$ of  size $\ell\le \nus$
   and a non-zero vector $w\in \RR^\ell\otimes \cMs$
   such that $L$ is positive definite  on
    $\cD_p$ and $L(X)w=0$. 
\end{lemma}

 Summarizing Lemma \ref{prop:quant} and
 combining Lemma  \ref{cor-sharp-EW-alt}
 with the argument behind lemma \ref{prop:quant}
 gives:

\begin{theorem}
 \label{thm:mainQuant}
 Suppose 
  $p$ is a symmetric $\dd\times\dd$ matrix-polynomial
 of degree  $d$ in $g$ variables which satisfies the conditions
 of Assumption \ref{assume}.
\begin{enumerate}[(i)]
\item
 There  
 is an $\ell \le \frac{\nu(\nu+1)}{2}$
 and $\ell \times \ell$ symmetric matrices
  $A_1,\dots,A_g$ such that
 $\cD_p=\cD_L$ where $L$
 is the monic  linear pencil
$
  L=I- \sum_j^{g} A_j x_j.
$
\item
  In the case that $p(0)=I_\dd$ the estimate on
  the size of the matrices $A_j$
  in  $L$ reduces  to $\frac{\nus(\nus+1)}{2}$,
  where $\nus= \dd \sum_0^\Nd g^j$.
\end{enumerate}
\end{theorem}

\subsection{Further Remarks}
\label{subsec:other}
\begin{remark}
\rm
 \label{anticipate}
   We anticipate that  the results  of this paper remain valid
   if symmetric free variables \index{symmetric variables}
   are replaced by free free variables. \index{free variables}
 \index{free variables}
   That is, with variables  $(x_1,\dots,x_g,y_1,\dots,y_g)$
   with the involution ${}^{\st}$ on polynomials
   determined by $x_j^{\st}=y_j$, $y_j^{\st}=x_j$, and, for polynomials
    $f$ and $g$
   in these variables, $(fg)^{\st}=g^{\st}f^{\st}.$ These polynomials
    are evaluated at tuples $X=(X_1,\dots,X_g)\in M_n(\mathbb R^g)$
   of $n\times n$ matrices with real entries.
   We do not see an obstruction to the free free variable 
   analog of Theorem \ref{thm:main}
   using the arguments here.  Indeed  arguments for 
   such variables are
    often easier than for symmetric variables.
\qed \end{remark}

\begin{remark}\rm
 \label{rem:need-semi-algebraic}
   Fix a positive integer $\mu$
   and let $\mathcal L$ denote a collection
  of monic linear pencils of size at most $\mu$.
   The matrix convex set $\mathcal C=\mathcal C(n)$ 
   defined by 
\[
  \mathcal C(n)=\{X\in\smatng: L(X) \succ 0 \mbox{ for all } L\in\mathcal L\},
\]
  has the following finiteness property.  If $X\in\smatng,$ then 
  $X\in\mathcal C(n)$ if and only if for every subspace 
  $\mathcal M$ of $\mathbb R^n$ of dimension $k\le \mu$,
  the tuple $P_{\cM}X|_{\cM}\in\mathcal C(k)$. On the other hand,
  this latter property does not suffice to guarantee 
  that $\mathcal L$ can be replaced by a finite collection
  of monic linear pencils.  Thus, some additional hypothesis,
  such as assuming 
 $\cD_p$ is determined by a polynomial,  is essential 
  to reach the conclusion of  Theorem 
  \ref{thm:main}.
\end{remark}

\section{The Case of Irreducible $p$}
 \label{sec:irreducible}
 The main result of this section,
   Theorem \ref{thm1} below,  says if $p$ satisfies 
  Assumption \ref{assume},  
  $p(0)=I_\dd$, and  $p$ is irreducible in a sense made
  precise below,  then $p$ has degree at most two.
 Moreover, under these
  assumptions and with $p$ scalar-valued ($\delta=1$), 
  Corollary \ref{thm:degreetwo}
  exhibits a very close connection
  between $p$ and an $L$ satisfying the conclusion of Theorem \ref{thm:main}.
  Recall, $p$ is a symmetric $\dddd$-matrix valued polynomial
  of degree $d$ in $g$ freely non-commuting variables.

  \begin{lemma}
 \label{lem2}
  Suppose $p \in \cPddd$
  satisfies the conditions of Assumption \ref{assume}.
  Suppose further that $p(0)=I_\dd$.
   If
 \begin{enumerate}[(i)]
  \item  $(X,v)\in \hbd_p(n);$ 
  \item  $L$ is  a monic  linear pencil  of size $\ell$ which
   is positive definite  on each $\cD_p(n)$; and
  \item there is a vector
   $0\ne w\in \ \RR^\ell\otimes \cMs,$
    where
   \[
    \cMs  =\{q(X)v: q\in \cPNd \},
   \]
    such that
    $L(X)w=0,$
 \end{enumerate}
  then there exists a non-zero $q\in\cPNpd$
  such that $q(X)v=0.$
     (Note: it is not  assumed that $L$ is the
   ``master monic  linear pencil''  from Theorem \ref{thmL}.)
\end{lemma}

\begin{proof}
  Write the monic  linear pencil $L$ as
 \[
   L=I+\sum A_j x_j,
 \]
  where the $A_j$ are $\ell\times \ell$ symmetric matrices.
  The tuple $X$ acts on $\mathbb R^n$ and
  hence $A_j\otimes X$ acts upon $\mathbb R^\ell\otimes \mathbb R^n$.
  With respect to this tensor product decomposition $w=\sum e_j \otimes h_j,$
  where $\{e_1,\dots,e_\ell\}$ is the standard orthonormal
  basis for $\mathbb R^\ell$ and $h_j\in\cMs$.
  From the definition of $\cMs$,
  there exists polynomials $r_j\in \cPNd$
  such that $h_j=r_j(X)v$.

  Since $L(X)w=0$, for each $m$  we have
  $0= [ e_m^{\star} \otimes I]  L(X)w.$ Thus,
 \begin{equation*}
  \begin{split}
    0   = &[ e_m^{\star} \otimes I ] [ w +\sum_k
                            \sum_j A_k e_j\otimes X_k r_j(X)v]\\
       = & [r_m + \sum_{k,j} (e_m^{\star}A_ke_j) x_k r_j](X)v.
  \end{split}
 \end{equation*}
   Now we argue, by contradiction,
   that the elements $q_m$ of $\cPNpd$ given by
 \begin{equation*}
   q_m(x) =  r_m(x) + \sum_{k,j} (e_m^{\star}A_ke_j) x_k r_j(x)
 \end{equation*}
   are not all $0$. If they were all $0$, then each
   $r_m$ satisfies $r_m(0)=0$;  that is, $r_m$ has no constant
   term. But, then, by the same reasoning, each
   $r_m$ has no linear terms and continuing along these lines
    we ultimately conclude  that all the $r_m$ are $0$.
   On the other hand, since $w\neq 0$, there is an
   $m$ such that $h_m=r_m(X)v\ne 0,$ a contradiction.
   Thus, 
   there is an $m$ such that $q_m\ne 0$ and at the
   same time $q_m(X)v=0$.  To complete the proof,
   observe that the degree of this $q_m$ is at most
   $\Nd+1$.
\end{proof}

\begin{remark}\rm
 \label{rlr1}
    Let $R$  denote the 
    element of $ \cP^{\ell\times \delta}$
     whose $m$-th row is the $r_m$
      produced
    in the proof of Lemma \ref{lem2}.
    The lemma
    says that $R$ is not zero. On the other hand,
    $R(X)v=w$ and $L(X)R(X)v=L(X)w=0.$  Hence
    the symmetric polynomial $R^{\star} L R$ is non-zero, but vanishes
    at $(X,v)$.
\qed \end{remark}

The polynomial $p$
 \index{minimum degree irreducible}
  \index{minimum degree defining polynomial}
 is a
 {\bf minimum degree irreducible}, or
 {\bf a minimum degree defining polynomial for $\cD_p$}, 
  provided the only $q\in \cP^\dd_{d-1}$ which satisfies
  $q(X)v=0$ for every $n$ and every $(X,v)\in\hbd_p(n)$ is $q=0$. 
  Note  that,
  while $p$ is restricted by Assumption \ref{assume}
  to be symmetric, the polynomial entries of $q$ need not be symmetric.
  Of course,  $q(X)$ is not symmetric (whenever $\delta>1$), 
  but rather an operator
  from $\mathbb R^\delta \otimes \mathbb R^n$ to $\mathbb R^n$.  


\begin{theorem}
 \label{thm1}
   If the polynomial $p \in \cPddd$
   satisfies Assumption \ref{assume} and if also $p(0)=I_\dd$,
   then there exists a non-zero  $q\in \cPNpd$
   such that $q(X)v=0$ for every $n$ and $(X,v)\in\hbd_p(n)$.

   In particular, if the graded set $\cD_p=(\cD_p(n))$ 
    is bounded and convex, 
   if  $p(0)=I_\dd,$ and if $p$ is
   a minimum degree defining polynomial
   for $\cD_p$, then the degree of $p$ is at most two.
\end{theorem}

\begin{proof}
   Given $(X,v)\in \hbd_p(n)$, let
 \begin{equation*}
   C_{(X,v)}=\{q \in \cPNpd : \ \ q(X)v=0 \}.
 \end{equation*}
   Note that $C_{(X,v)}$ is a subspace of $\cPNpd$.

   Let $\cMs=\{r(X)v: r\in\cPNd\}$.
   By Proposition \ref{cor-sharp-EW-alt} there is
   a monic  linear pencil $L$ of some size
   $\ell\le \nus$ ($\nus$ is defined 
   in Equation \eqref{eq:def-nus})
   such that $L$ is positive definite
   on $\cD_p$ and a non-zero vector
   $w\in \mathbb R^\ell\otimes \cMs$
   such that $L(X)w=0$.
   Thus Lemma \ref{lem2} applies to produce a non-zero
   $q\in\cPNpd$ such that $q(X)v=0$.
   Hence $C_{(X,v)}$ is non-trivial (not $(0)$).

   Given a positive integer $s$
   and  $(X^j,v^j) \in \hbd_p(m_j)$ for $1\le j\le s$, 
   let $(W,u)=\oplus (X^j,v^j)$.  Then $(W,u)\in\hbd_p(m)$,
   where $m=\sum m_j$. Further, 
   by what has already been proved,
   there exists
   a non-zero $q\in \cPNpd$
   such that $q(W)u=0$. But then $q(X^j)v^j=0$ for each $j$.
   Hence $q\in \cap_{j=1}^{\ell} C_{(X^j,v^j)}.$
   It follows that the collection of subspaces
   $C_{(X,v)}$ is closed with respect to finite
   intersections.
   Since also each $C_{(X,v)}$ is a non-trivial subspace of
   the finite dimensional space $\cPNpd$,
   there is a smallest (and non-trivial) subspace $C_{(Y,w)}$
   uniquely determined by the condition that
   it has minimum dimension. Note that any (non-zero)
   $q\in C_{(Y,w)}$ must vanish on all of $\hbd_p$,
   since if $(X,v)\in \hbd_p$ and $q(X)v\ne 0$, then
   $C_{(X,v)}\cap C_{(Y,w)} \subsetneq C_{(Y,w)}$.

   The second part of the theorem follows immediately from
   the first part and the definition of minimum degree defining
   polynomial.
\end{proof}

\begin{corollary}
 \label{thm:degreetwo}
   Suppose $p \in \cPddd$ satisfies the conditions
   of Assumption \ref{assume},    $p(0)=I_\dd$, and 
    $p$ is a minimum degree defining polynomial for $\cD_p$.
%
   If $\dd=1$, there exists a $1\times 1$
   monic  linear pencil $L_0,$ an integer
   $m\le g$  and an $m\times 1$   linear
   pencil ${\hat L}$ with ${\hat L}(0) =0$
   such that $\cD_p=\cD_L$, where
 \[
    L=\begin{pmatrix} I_m & \hL \\ \hL^{\st} & L_0
        \end{pmatrix}.
 \]
  In fact, $p$ is the Schur complement of the $(1,1)$
  entry of $L$; i.e.,
 \[
   p = L_0 -\hL^{\st} \hL.
 \]
\end{corollary}

 This corollary of Theorem \ref{thm1}
 is, for the most part, an improvement over the  main result of \cite{DHM07}.
 In particular, the result here
 removes numerous hypotheses found in \cite{DHM07} while reaching
 a stronger  conclusion, though here it is assumed that
 $\cD_p$ is convex, rather than the
 weaker condition that $\ocD_p $ is convex.
 The techniques here are completely different than
 those in \cite{DHM07}.

\begin{proof}
  The first part of Corollary \ref{thm:degreetwo} is covered
  by Theorem \ref{thm1}.  It remains to prove if
  $p$ is a symmetric  free polynomial in $\cP^{1 \times 1}_2$,
  if $p(0)=1$  and if $\cD_p$ is both bounded and convex,
  then $p$ has the form
 \[
   p= 1 +\ell(x) -\sum_{j=1}^g \lambda_j(x)^2,
 \]
   where $\ell$ and each $\lambda_j$ are linear.

   Since $p$ has degree two and is symmetric, there is a
   uniquely determined symmetric $g\times g$ matrix $\Lambda$ such that
 \[
   p(x)=1 +\ell(x) -  \langle \Lambda x,x\rangle,
 \]
   where $x$ is the vector with entries $x_j$.
   If $\Lambda$ is not positive semi-definite,
   then there is a $t\in\mathbb R^g$ such that
   $ \langle \Lambda t,t\rangle  <0$ and hence,
   for $s\in\mathbb R$,
 \[
   p(st)= 1 + s\ell(t) - s^2 \langle \Lambda t,t \rangle
 \]
   is either positive for all $s\ge 0$ or is
   positive for all $s\le 0$ depending upon the
  sign of $\ell(t)$.  In either case,
  $\cD_p(1)$ is not bounded.  Thus the boundedness of $\cD_p$ implies
   that $\Lambda$ is positive semi-definite.
   Hence there is an $0\le m\le g$
  and an  orthogonal
   set of vectors $u_1,\dots,u_g$ such that
 \[
   \Lambda = \sum_1^m u_\ell u_\ell^{\star}.
 \]
   Letting $\lambda_\ell = \sum_j (u_\ell)_j x_j$,
 \[
   \hL=\begin{pmatrix} \lambda_1 \\ \vdots \\ \lambda_m\end{pmatrix},
 \]
  and $L_0=1+\ell$ the conclusion of the corollary
  follows. 
\end{proof}

   The following example shows that
   Corollary \ref{thm:degreetwo} requires the irreducibility hypothesis.
   Here we work with two variables $(x,y)$. Let
   $b(x,y)=1-x^2-y^2$ and $f(x,y)=1-(x-\frac14)^2-y^2$.
   The set
 \[
   \cD=\cD_{b\oplus f}=\{(X,Y): b(X,Y) \succ 0,\ f(X,Y)\succ 0\}
 \]
   is convex.
   Let  $p_1 = fbf$ and $p_2=bfb$.  Then
  $\cD_{p_1}=\cD=\cD_{p_2}$.  Hence,
  neither $p_1$ nor $p_2$ is a minimum degree
  defining polynomial for $\cD$. Indeed, 
  $bf$ vanishes on $\hbd_{p_1}$ and $fb$
  on $\hbd_{p_2}$.  On the other hand,
  neither $bf$ nor $fb$ is a symmetric so neither
  is a candidate for a minimum degree defining
  polynomial.  It is likely that in this
  example there does not exist a minimum degree
  defining polynomial for $\cD$.

\section{Free Real Algebraic Geometry}
 \label{sec:freesag}
     One of the main branches of real algebraic geometry,
    dating back to Hilbert, 
    is semi-algebraic geometry, a  subject which
deals with polynomial inequalities.
 Free (non-commutative) semi-algebraic
geometry has been developing for about a decade.

 This section describes implications of
  the LMI representation of 
 Theorem \ref{thm:main} for 
 free semi-algebraic geometry.
 It also contains a strengthening of 
Theorem \ref{thm:main}.
Another area of contact is semi-definite programming (SDP),
one of the main developments  in optimization
over the last two decades.   We state one
(disturbing) result
in the language of SDP in  \S \ref{sec:persp}.
   
\subsection{Free Semi-Algebraic Sets}

This subsection gives definitions of free semi-algebraic sets 
and their principal components.
 Recall, from  Subsection \ref{subsec:semi-algebraic}, that
  $pc[\cW]$ denotes
  the \df{principal component} of a graded set
  $\cW\subset \smatg$.  Also, if $p$ is a matrix-valued
  symmetric polynomial and $p(0)$ is invertible, then 
  $pc[\pos] =\cD_p.$

\begin{lemma}
 \label{lem:capbasic}
   If $p_j \in \cP^{\dd_j \times \dd_j}$ 
   is symmetric and $p_j(0)$ is  invertible for $j=1,2,\dots,s,$ then
  \beq
\label{eq:capLots}
   \cap_1^s \posI_{p_j} =\pos
   \qquad  and \qquad 
   \cap_1^s \cD_{p_j} \supset \cD_p, 
\eeq
  where $p=\oplus p_j$. 
 Further
\beq
\label{eq:pcDp}
pc[  \cap_1^s \cD_{p_j}] = \cD_p.
\eeq
\end{lemma}

\begin{proof}
  The intersection property of $\posI$  is obvious, 
  as is the inclusion, 
  $\pos \subset \posI_{p_j},$ for each $j$. 
   Hence $\cD_p= pc[\pos] \subset pc[ \posI_{p_j}]=\cD_{p_j}$,
   so
   $\cap_1^s \cD_{p_j} \supset \cD_p$ and 
   $$
   pc[\cap_1^s \cD_{p_j}] \supset \cD_p.
   $$
  Since $\cD_{p_j} \subset \posI_{p_j}$, we have
   $\cap_1^s \cD_{p_j} \subset \cap_1^s \posI_{p_j}
    = \pos$. 
  Consequently, 
 \[
   pc[\cap_1^s \cD_{p_j}] \subset pc[\cap_1^s \posI_{p_j}]
    = pc[\pos]=\cD_p. 
 \]
\end{proof}

\def\cpdj{{ \cP^{\delta_j \times \delta_j} } }
 Classically, a basic open semi-algebraic set is a set of the form
\[
  S=\{x\in\mathbb R^g : \ttq_j(x)>0, \ \  j=1,\dots,\sigma\},
\]
  for given (commutative) polynomials $\ttq_j$  \cite{BCR98}. 
 There are several natural ways to extend this definition to free 
  $*$-algebras.
 The one which follows has the property that theorems
 flowing from it are stronger than analogous theorems
 using other definitions.
 Paralleling classical real algebraic geometry, we define a
 {\bf  \ncbo  semi-algebraic set} (containing 0)
  to be  a graded
  set of the form $\cap_j \cD_{p_j}$
  for some finite set of symmetric matrix polynomials $p_j$ in $\cpdj$ with
  $p_j(0)$ invertible. 
  Note, while each $\cD_{p_j}$ is a connected set, the
  intersection need not be. 
 A \df{free open semi-algebraic set}  (containing 0)
 is a finite union of  \ncbo  semi-algebraic sets. 
 
    In  classical real algebraic geometry, 
   the components of a semi-algebraic
   set are themselves  semi-algebraic.
   Lemma \ref{lem:capbasic} says that the component
  of $0$ of a \ncbo semi-algebraic set is again
  \ncbo semi-algebraic; and 
   corollary \ref{cor:morebasic-rad}
   gives natural  conditions under which 
   the principal component of a 
   free open  semi-algebraic set is itself
   a \ncbo  semi-algebraic set.

 This section develops some properties of free open semi-algebraic sets, 
 several of which contrast markedly with the classical
 situation. These properties lead to a strengthening of 
  Theorem \ref{thm:main} 
 and they are used to show that if the projection of an LMI
 representable set  is a free open semi-algebraic set, then it
  is in fact LMI representable.  

\subsection{Connectedness}
   Before turning to free semi-algebraic
  sets, this subsection derives some fairly general facts with
  the theme of connectedness.

\begin{proposition}
 \label{prop:basic-rad}
  Suppose  $p_j\in\cP^{\dd_j\times \dd_j}$
  for $j=1,2,\dots,s,$ each $p_j$ is symmetric,  and each 
  $p_j(0)$ is  invertible.  
  Further, suppose  $\mathcal W=(\mathcal W(n))$
  is a graded set with $\mathcal W(n) \subset \cup_{j=1}^s  \posI_{p_j}(n)$
  for each $n;$
   that is, $\mathcal W\subset \cup_1^s \posI_{p_j}$. 
  If $\mathcal W$ 
  respects direct sums and each $\mathcal W(n)$
  contains $0$ and is open and connected,
  then there is a $k$
  such that $\mathcal W\subset \cD_{p_k}$;  that is,
  $\mathcal W(n)\subset \cD_{p_k}(n)$ for each $n$. 
\end{proposition}

\begin{proof}  We begin by proving, if $X\in\mathcal W(n)$
  and if $X(t)$ is
  a (continuous) path  for $0\le t\le 1$ such that
  $X(0)=0,$ $X(1)=X,$ and $X(t)$ lies in $\mathcal W(n)$,
  then there is a $j$
  such that $p_j(X(t))$ is invertible for every $0\le t\le 1$.

  Arguing by contradiction,  suppose no such $j$ exists.
  Then for every
  $1\le \ell \le N$ 
  there exists a $0\le t_\ell \le 1$ such that $p_\ell(X(t_\ell))$
  is not invertible.  
  Since $\mathcal W$
  is closed with respect to direct sums, 
  $Z=\oplus X(t_\ell) \in\mathcal W(nN)$.
  It follows that there is some $1\le j\le N$ such that
  $Z\in\posI_{p_j}(nN)$ and in particular, 
  $p_j(Z)$ is invertible, contradicting
  $p_j(X(t_j))$ not invertible.  We conclude that there is
  some $j$ such that  $p_j(X(t))$ is invertible for $0\le t \le 1$
  and hence $X(t)\in\cD_{p_j}$ for all $0\le t\le 1$.

  Now suppose there is an $m$ and a  $Y \in\mathcal W(m)$ such that
  $ Y\notin \posI_{p_s}(m)$.
  In particular, $p_s(Y)$ is not invertible.
  Since $Y$ is in $\mathcal W(m)$, there is a continuous
  path $Y(t)\in\mathcal W(m)$ such that $Y(0)=0$ and $Y(1)=Y.$
  Now let $n$ and $X\in\mathcal W(n)$ be given.  There is a continuous
  path $X(t)\in\mathcal W(n)$ with $X(0)=0$ and $X(1)=X$.
  Let $Z(t)=X(t)\oplus Y(t)$;
  which is in $\cW(n+m)$ since $\cW$ respects direct sums.
  Thus
  $Z(t)\in\mathcal W(n+m)$ is a continuous path ($0\le t\le 1$)
  with $Z(0)=0$. From what has already been proved,
  there is a $j$ such that $p_j(Z(t))$ is invertible for
  each $0\le t\le 1$. Thus $p_j(Y)$ is invertible
  and we conclude that $j\ne s.$
  At the same
  time $p_j(X(t))$ is invertible for $0\le t\le 1$
  and thus $X\in\cD_{p_j}$.
   Hence $X\in \cup_{1}^{s-1}\cD_{p_j}(n)$.
  We have proved: either $\mathcal W(m)\subset \posI_{p_s}(m)$
  for every $m$,  or
  $\mathcal W(n)
\subset \cup_{1}^{s-1} \cD_{p_j}(n)
\subset \cup_{1}^{s-1} \posI_{p_j}(n)$ for every $n$.
Since $\cW$ is connected and contains $0$,
the first alternative becomes $\cW$ is a subset of $\cD_{p_s}$;
   that is,  $\cW(m)\subset \cD_{p_s}(m)$ for every $m$. 
Induction
  now finishes the proof.
\end{proof}

\begin{corollary}
 \label{cor:morebasic-rad}
  Suppose $p_{k,j}$ is a finite
  collection ($k=1,\dots,t; \  j=1,\dots,s_k$) 
  of symmetric matrix-valued 
  free polynomials with $p_{k,j}(0)$
  invertible. 
  Suppose  the graded set $\mathcal W=(\mathcal W(n))$
  has the form
   \begin{enumerate}[(i)]
   \item
   \label{eq:WisD}
 $$  \cW = pc[ \cup_{k=1}^t \cap_{j=1}^{s_k} \cD_{p_{k,j}} ],$$
   or 
\item
\label{eq:WisI}
$$\cW = pc[\cup_{k=1}^t \cap_{j=1}^{s_k} \posI_{p_{k,j}}].  $$
\end{enumerate}
 If $\cW$  
  respects direct sums, 
  then there is a $k_0$ such that $\cW =\cD_{p^{k_0}}$,
  where $p^k$ is defined by
  $$p^{k}=\oplus_{j=1}^{s_k} p_{k,j}.$$ 
  \end{corollary}

\begin{proof}

Either of the hypotheses 
\eqref{eq:WisD} or \eqref{eq:WisI}
imply that 
$$
\cW \subset  \cup_{k=1}^t \cap_{j=1}^{s_k} \posI_{p_{k,j}}=
  \cup_{k=1}^t \posI_{p^k},
  $$
the equality holding because of Lemma \ref{lem:capbasic}.
 Proposition \ref{prop:basic-rad} implies
  there is a $k_0$ such that
  $ \cD_{p^{k_0}} \supset \cW$.
  Because of this containment,
  hypothesis \eqref{eq:WisD} and Lemma \ref{lem:capbasic}, we have
   $$
    \cD_{p^{k_0}}\supset
   \cW = pc[\cup_{k=1}^t \cap_{j=1}^{s_k} \cD_{p_{k,j}}]
   \supset
  pc[ \cup_{k=1}^t  \cD_{p^{k}}] \supset \cD_{p^{k_0}}.
   $$
   Thus
  $\cW=  \cD_{p^{k_0}}$.
  Hypothesis \eqref{eq:WisI} along with 
  $\posI_p \supset \cD_p$ and 
  $ \cD_{p^{k_0}} \supset \cW$ imply
    hypothesis \eqref{eq:WisD} holds.
    So, under either hypothesis, the required conclusion follows. 
    \end{proof}

\subsection{Free Semi-Algebraic Sets vs. Basic Ones}  
  Corollary \ref{cor:morebasic-rad} \eqref{eq:WisD} 
  rephrased in terms of semi-algebraic sets
  gives the following result. 

  \begin{corollary}
\label{prop:basic1}
  Let $\mathcal W=(\mathcal W(n))$ 
  be a graded set which is 
  contained in (resp. is the principal component of)
    a free open  semi-algebraic set. 
  If $\mathcal W$ 
  respects direct sums and each $\mathcal W(n)$
  contains $0$ and is open and connected,
  then $\cW$ is contained in (resp. equals) 
  the principal component of 
   some  \ncbo semi-algebraic set.  
\end{corollary}

  Theorem \ref{thm:main} can now be strengthened
  as follows. 
  
\begin{theorem}
 \label{cor:main-cor}
 Suppose the graded set 
 $\mathcal W$   is bounded and matrix convex.
 \begin{enumerate}[(i)]
 \item
 \label{it:fob}
   If $\mathcal W$ is the principal component of 
   a  free open semi-algebraic set,  \ or
 \item
 \label{it:fobI}
   if $\cW$ is the principal component of 
   a graded set of the form
$$\mathcal  \cup_{k=1}^t \cap_{j=1}^{s_k} \posI_{p_{k,j}}, $$
\end{enumerate}
then $\cW$ has an LMI representation.
\end{theorem}

\begin{proof}
   Because $\cW$ is matrix convex, it is closed with respect
   to direct sums.  Thus, under either hypothesis 
   \eqref{it:fob} or \eqref{it:fobI} , Corollary 
  \ref{cor:morebasic-rad} implies that $\cW$
   has the form $\cD_p$ for some symmetric matrix-valued $p$.
   Further, $\cD_p$ is convex and thus Theorem \ref{thm:main}
    implies that $\cD_p$ has an LMI representation.
\end{proof}

   This theorem  implies that
   the principal component of a free open semi-algebraic set
   is itself free semi-algebraic under the additional hypothesis
   that it is matrix convex.

\subsection{Free Projections}
 \label{subsec:proj}
  One of the  key facts in real algebraic geometry is that 
  the projection of a semi-algebraic set is
  again semi-algebraic (by Tarski's principle).  
  Thus, if $S\subset \mathbb R^{g+h}$
  is an open semi-algebraic set, then  
  the projection onto its first $g$ coordinates is a semi-algebraic set. 
   Given a graded subset $\cD=(\cD(n))$ of
  the graded set $\mathbb S(\mathbb R^{g+h})$,
  the \df{(free) projection} \index{projection, free}
  of $\cD$ (onto $(\smatng)$) 
  is the graded set $\pi(\cD)=(\pi(\cD)(n))$ defined by
 \[
   \pi(\cD(n)) = \{ X\in\mathbb S_n(\mathbb R^g):
       \mbox{ there is a } Y\in\mathbb S_n(\mathbb R^h)
      \mbox{ such that } (X,Y)\in\cD(n) \}.
 \]

\begin{lemma}
 \label{freeprojections}
   The following properties are inherited under free projections. 
\begin{enumerate}[(i)]
 \item respects  direct sums;
 \item respects unitary conjugation;
 \item openness;
 \item connectedness;
 \item boundedness;  and  
 \item matrix convexity.
\end{enumerate}
\end{lemma}

\begin{proof}
 Straightforward.
\end{proof}

An immediate consequence of
 combining  this lemma with  Theorem \ref{cor:main-cor}
 \eqref{it:fob} is a fact which is far from what one finds
in the classical commutative case.   
  
\begin{corollary}
 \label{cor:semi-project} 
 If the graded subset $\cW$ of $\mathbb S(\mathbb R^{g+h})$
 is bounded and 
has an LMI representation and if its projection $\pi(\cW)$ 
   is  a free open semi-algebraic set,
  then  
 $\pi(\cW)$  has an LMI representation.
 
\end{corollary}

 This corollary plus the example in the following subsection shows that
 the projection of a free bounded basic open semi-algebraic
 set need not be free  open semi-algebraic. 
We state this as a proposition, since it is so contrary to a basic tenet of
classical real algebraic geometry.

\begin{proposition}
 \label{prop:nofreeTarski}
  There exists a monic linear pencil $L$ in $g+h$ variables
  such that the projection $\pi(\cD_L)$ is neither of the form 
  \eqref{it:fob}  nor \eqref{it:fobI}
  in Theorem \ref{cor:main-cor}.
   In particular,
   there exist convex  \ncbo  semi-algebraic sets with 
   projections which are not  free open semi-algebraic. 
\end{proposition}

 To establish the proposition, it suffices 
 (thanks to Corollary \ref{cor:semi-project},) to 
 produce a monic linear pencil $L$ in $g+h$ 
 variables with the property that $\pi(\cD_L)$
 is not of the form $\cD_M$ for a monic linear
 pencil $M$ in $g$ variables.  
  The following is an example of such an $L$.

\subsection{The TV Screen - an Example}
 \label{subsec:example}

 \rm
  Consider the set
 \[
   S=\{(x_1,x_2)\in\mathbb R^2 : 1-x_1^4-x_2^4 >0\},
 \]
 often called the TV screen.
  This set is evidently convex.
  By the  line test in
  \cite{HV07}, there does not exist a monic linear pencil
  $L$ such that $\cD_L(1)=S$.  Thus, 
  {\it if $\mathcal T=(\mathcal T(n))$
  is a graded set with $\mathcal T(1)=S$, then $\mathcal T$
  does not have an LMI representation. }

 Now we build a certain type of representation for $S$. 
 Given $\alpha$ a positive real number,
 choose
 $\gamma^4= 1+2\alpha^2$ and
  let
\begin{equation}
 \label{eq:Lalpha0}
  L^\alpha_0 = \begin{pmatrix} 1 & 0 &  y_1 \\
             0 & 1 & y_2 \\
              y_1 & y_2 &
                  1- 2\alpha (y_1+y_2) \end{pmatrix}
\end{equation}
 and
\begin{equation}
 \label{eq:Lalpha12}
  L^\alpha_j =\begin{pmatrix} 1 & \gamma x_j
        \\ \gamma x_j & \alpha + y_j \end{pmatrix}, \quad j=1,2.
\end{equation}
  Note that the $L^\alpha_j$ are not monic, but because $L^\alpha_j(0)\succ 0,$
  they 
  can be normalized to be monic without altering the solution
  sets of $L^\alpha_j(X)\succ 0$.
     That
 \[
   S =\{(x_1,x_2) \in \RR^2: \mbox{there exists }(y_1,y_2) \mbox{ such that }
    L_j^\alpha(x,y)\succ 0, \ \ j=0,1,2\},
 \]
 follows by taking Schur complements
 and a bit of algebra which shows
   $$
   S=\{ (x_1, x_2) : \ 1-2\alpha(y_1+y_2) - y_1^2-y_2^2 >0  , \; \ 
     \alpha + y_j > \sqrt{1+2\alpha^2} \; x_j^2 \}.
   $$
  {\it Consequently,  $S=\pi(\cD_{L^\alpha} (1))$, where 
  $L^\alpha= L_0^\alpha \oplus L_1^\alpha \oplus L_2^\alpha$.}
  
  Note that 
  choosing $\alpha=0$ gives the representation 
  of the TV screen $S$ often found in the
  literature. It is not satisfactory for the
  present purposes, since it 
  can not be normalized to be monic.   
  \qed

 \noindent
{\it Proof  of Proposition \ref{prop:nofreeTarski}
}
  As was seen in the example above, for $\alpha>0$ fixed, $\pi(\cD_{L^\alpha})$ 
  is the projection of an LMI representable set.
  However, it 
  does not have either of the
  forms 
   \eqref{it:fob}  
   or \eqref{it:fobI}
  given in the proposition, as otherwise it would have,
  by  Theorem \ref{cor:main-cor}, 
  an LMI representation,
  thereby contradicting  paragraph one of the example.
\qed

\subsection{Outside Perspectives}
\label{sec:persp}
Here we include two remarks aimed at readers
with interest in either semi-definite programming
 or free real algebraic geometry.

\begin{remark}
\rm
The paradigm problem in semi-definite programming
is to maximize a linear functional over an
SDP representable set.
  A subset $C\subset \mathbb R^g$ is called 
  semi-definite 
  programming representable  or \df{SDP representable},
  if there is a monic linear pencil $L$ in $g+h$ variables 
  such that
$ C=\pi(\cD_L(1))$.
For a general survey and overview of semi-definite programming,
 see Nemirovski's Plenary Lecture   at the 2006 ICM  \cite{Ne06}.

 
  By analogy with the scalar commutative case, 
  a graded subset $\cC=(\cC(n))$
  of $\smatg$   is {\bf freely SDP representable} if
  there is a monic linear pencil $L$ in $g+h$
  variables such that $\cC(n)=\pi(\cD_L(n))$ for each $n$. 
  For example, 
  the graded 
  set $\pi(\cD_{L^\alpha})$ has, by construction,
  a free SDP representation.  
  In this terminology, Corollary \ref{cor:semi-project}
  says  
  \begin{quote}
  {\it if $\cC$ is both
  SDP representable and free semi-algebraic, then
  $\cC$ is LMI representable. }
  \end{quote}
    \qed
 \end{remark}

\begin{remark}
\rm
 As mentioned earlier  
 there are several other natural choices of the notion of
 free semi-algebraic set beyond the one adopted
 earlier.  
 Here we mention one.
  Given a symmetric $p\in\cP^{\dddd}$ with $p(0)\succ 0$ (not
  just invertible), let 
\[
  \mfP_p (n) =\{X\in\smatng : p(X) \succ 0\}.
\]
 Observe that $\cap \mfP_{p_j} = \mfP_p$, where $p=\oplus p_j$. 
 The lemmas and theorems  of this section, appropriately modified,
 hold if  $\mfP_p$ is used as the notion of a \ncbo semi-algebraic set. 
  For example,
  if       
      $pc[\cup_{k=1}^s \mathfrak{P}_{q_k}]$
  is bounded and matrix convex, then it 
   has an LMI representation and is thus
   a \ncbo semi-algebraic set.
  \qed
 \end{remark}

We thank Jiawang Nie for raising the issue
of projected matrix convex sets and we thank Igor Klep
and Victor Vinnikov for fruitful discussions of
the TV screen example(s) above.



\newpage

\centerline{NOT FOR PUBLICATION}

\section{Systems Theory Motivation}
 \label{subsec:motivate}
One of the main advances in systems engineering in the 1990's
was the conversion of a set of problems to LMIs,
since LMIs, up to modest size, can be solved numerically
by semi-definite programs \cite{SIG97}.
A large  class of linear
systems problems are described in terms of
a signal flow diagram $\Sigma$ plus $L^2$
constraints (such as energy dissipation).
Routine methods convert such problems into
a free polynomial inequalities
of the form  $p(X)\succeq 0$ or $p(X)\succ 0$.

Instantiating
specific  systems of linear differential
equations  for the "boxes" in the system flow diagram
amounts to substituting  their coefficient matrices
for variables in the polynomial $p$.
Any property  asserted to  be true must hold
when matrices of any size are substituted into $p$.
 Such problems are referred to as dimension free.
We emphasize, the polynomial $p$ itself is determined
by the signal flow diagram $\Sigma$.

Engineers vigorously seek convexity, since optima are global
and convexity lends itself to numerics.
Indeed,  there are over  a thousand papers trying to convert
linear systems problems to convex ones and the only known technique is
the rather blunt trial and error instrument of trying to guess an LMI.
Since  having an LMI is seemingly  more restrictive than convexity,
there has been the hope, indeed expectation, that some practical
class of convex situations   has been missed.
  The problem solved here
  (though not operating at full engineering generality,
  see \cite{HHLM08})
  is a paradigm for the type of algebra occurring
  in systems problems governed by signal-flow diagrams;
 such physical problems directly present free semi-algebraic sets.
  Theorem \ref{thm:main} gives compelling evidence
  that all such convex situations are associated to some LMI.
Thus we think the implications of our results here are negative
for linear systems engineering; for dimension free problems
there is no convexity beyond LMIs.

It is informative to view this paper in  the context of semi-definite
programming, SDP.
Semidefinite programming, which solves LMIs up to modest size,
was  one of the main developments in optimization over the
  previous two decades.
  Introduced about 15 years
   ago \cite{NN94} it has had a substantial effect in
   many areas of science and mathematics;  e.g.,
  statistics, game theory,  structural design
  and computational real algebraic geometry,
with its largest impact likely being in control systems
and combinatorial  optimization.
   For a general survey, see Nemirovski's Plenary Lecture
   at the 2006 ICM, \cite{Ne06}.
 An introduction of  SDP techniques into a variety of areas
being pursued today was first given (and is well explained in)
 \cite{P00}.
   The numerics of semi-definite programming
   is well developed   and there are numerous packages; e.g.,
   \cite{Sturm99} \cite{GNLC95}  and comparisons \cite{Mi03}
     which apply
   when the constraint  is input as the solution to
    a Linear Matrix Inequality.

  A basic question regarding the range of   applicability of SDP is:
  which sets have an LMI representation?
  Theorem \ref{thm:main} settles, to a reasonable extent, the case where
 the variables are free
 (effectively dimension free matrices).

 For perspective, in the commutative case
  of a basic semi-algebraic subset $\cC$
  of $\RR^{g}$, as we have already mentioned,
  there is a stringent condition, called the
 ``line test'', which, in addition to convexity,
  is  necessary for $\cC$ to have an LMI representation.
  In two dimensions the line test is necessary and sufficient,
  \cite{HV07}. This was seen by Lewis-Parrilo-Ramana \cite{LPR05}
  to settle a 1958 conjecture of Peter Lax on hyperbolic polynomials
  and indeed LMI representations are closely tied to properties
  of hyperbolic polynomials.

In summary,  a (commutative)
bounded basic open semi-algebraic convex set
has an LMI representation, then it must pass
the highly restrictive line test; whereas
a  \ncbbo semi-algebraic set has
 an LMI representation if and only if it is convex.

\newpage

\tableofcontents

\newpage

\printindex

\end{document}